\newcommand\bbR{\mathbb{R}}
\newcommand\bbS{\mathbb{R}^{d\times d}_{\text{sym}}}

\newcommand\cC{\mathcal{C}}
\newcommand\cD{\mathcal{D}}
\newcommand\cF{\mathcal{F}}
\newcommand\cH{\mathcal{H}}
\newcommand\cP{\mathcal{P}}
\newcommand\cT{\mathcal{T}}

\newcommand\cX{\mathcal{X}}

\newcommand\bn{\boldsymbol{n}}
\newcommand\bF{\boldsymbol{f}}
\newcommand\bu{\boldsymbol{u}}
\newcommand\bv{\boldsymbol{v}}
\newcommand\bw{\boldsymbol{w}}

\newcommand\be{\boldsymbol{e}}
\newcommand\bx{\boldsymbol{x}}

\newcommand\bz{\boldsymbol{z}}
\newcommand\bp{\boldsymbol{p}}
\newcommand\bq{\boldsymbol{q}}

\newcommand\br{\boldsymbol{r}}
\newcommand\bs{\boldsymbol{s}}
\newcommand\beps{\boldsymbol{\varepsilon}}
\newcommand\bsig{\boldsymbol{\sigma}}
\newcommand\btau{\boldsymbol{\tau}}
\newcommand\bnabla{\boldsymbol{\nabla}}

\newcommand\bchi{\boldsymbol{\chi}}

\newcommand\ubu{\underline{\boldsymbol{u}}}
\newcommand\ubv{\underline{\boldsymbol{v}}}

\newcommand\mt{\mathtt{t}}

\newcommand{\sep}{\mathrel{;}}

\newcommand{\tr}{\operatorname{tr}}
\renewcommand{\div}{\operatorname{div}}
\newcommand{\bdiv}{\operatorname{\mathbf{div}}}

\documentclass[11pt, leqno]{article} 
\usepackage[utf8]{inputenc} 
\usepackage[T1]{fontenc}
\usepackage[english]{babel} 

\usepackage{lmodern} 
\usepackage{stmaryrd} 

\usepackage{amsmath,amsthm,amsfonts,amssymb}

\usepackage{ninecolors}
\usepackage{graphicx}
\usepackage{tabularray}
\usepackage{caption, subcaption}
\usepackage{tikz}

\usepackage{xfrac} 

\usepackage{stmaryrd}

\newcommand{\mmean}[1]{\left\{\kern-1.ex\left\{ #1 \right\}\kern-1.ex\right\}}

\usepackage{mathtools}
\DeclarePairedDelimiterX\norm[1]\lVert\rVert{
   \ifblank{#1}{\:\cdot\:}{#1}
}
\DeclarePairedDelimiterX\abs[1]\lvert\rvert{
   \ifblank{#1}{\:\cdot\:}{#1}
}
\DeclarePairedDelimiter{\inner}{(}{)}
\DeclarePairedDelimiter{\set}{\{}{\}}
\DeclarePairedDelimiter{\dual}{\langle}{\rangle}
\DeclarePairedDelimiter{\stack}{[}{]}

\usepackage{siunitx}

\usepackage[singlespacing]{setspace} 

\usepackage{titlesec}
\titleformat{\section}[block]{\filcenter\bfseries}{\thesection.}{1em}{}
\titleformat{\subsection}[block]{\bfseries}{\thesubsection.}{1em}{}

\newtheoremstyle{paper}{}{}{\itshape}{}{\scshape}{.}{.5em}{}
\theoremstyle{paper}
\newtheorem{theorem}{Theorem}
\newtheorem{proposition}{Proposition}
\newtheorem{lemma}{Lemma}

\newtheorem{remark}{Remark}

\usepackage[numbers]{natbib}
\usepackage{natbib}
\setcitestyle{aysep={}}

\AtBeginEnvironment{thebibliography}{\setstretch{1.1}}

\usepackage{fancyhdr} 

\usepackage[pdftex, hidelinks, hypertexnames = false, hyperfootnotes = false,
pdfpagemode = UseNone, pdfdisplaydoctitle = true]{hyperref}%

\usepackage[a4paper, tmargin=3cm, bmargin=3cm, rmargin=2.2cm, lmargin=2.2cm]{geometry}

\begin{document}
\title{\huge Hybridizable Discontinuous Galerkin Methods for Thermo-Poroelastic Systems}
\author{Salim Meddahi
	\thanks{This research was  supported by Ministerio de Ciencia e Innovación, Spain.}}

\date{}


\maketitle

\begin{abstract}
\noindent
We propose a high-order hybridizable discontinuous Galerkin (HDG) formulation for the fully dynamic, linear thermo-poroelasticity problem. The governing equations are formulated as a first-order hyperbolic system incorporating solid and fluid velocities, heat flux, effective stress, pore pressure, and temperature as state variables. We establish well-posedness of the continuous problem using semigroup theory and develop an energy-consistent HDG discretization. The method exploits computational advantages of HDG—including locality and static condensation—while maintaining energy conservation for the coupled system. We establish an $hp$-convergence analysis and support it with comprehensive numerical experiments, confirming the theoretical rates and showcasing the method's effectiveness for thermo-poroelastic wave propagation in heterogeneous media.
\end{abstract}
\bigskip

\bigskip



\section{Introduction}\label{s:introduction}

The coupled interaction of mechanical deformation, pore-fluid flow, and heat transport governs numerous engineering and geophysical applications, including geothermal energy extraction, nuclear waste storage, hydraulic fracturing, and enhanced oil recovery \cite{ayub,fanfeng, putra}. Accurate prediction of these phenomena presents significant challenges because the governing thermo-poroelastic system involves multiple time scales and sharp solution gradients arising from material heterogeneity or localized sources.

Classical computational methods such as finite difference \cite{gaspar2003}, finite element \cite{santos1986II}, boundary element \cite{chen1995}, finite volume \cite{Lemoine2013}, and mixed methods \cite{Lee2023} typically require extremely fine meshes to resolve these features adequately. This computational burden has motivated the development of high-order methods, including spectral techniques \cite{morency} and high-order space-time Galerkin formulations \cite{bause2024}. In particular, advanced polytopal discontinuous Galerkin schemes developed in \cite{antonietti2023discontinuous, antoniettiIMA, bonetti2024, bonetti2023numerical} provide local conservation, geometric flexibility, and natural $hp$-adaptivity. However, their large computational stencils may result in a high cost per degree of freedom.

Hybridizable discontinuous Galerkin (HDG) methods \cite{Cockburn2009} address this computational challenge by statically condensing element unknowns onto faces, thereby producing smaller global systems while preserving the accuracy of standard DG methods. HDG has demonstrated remarkable effectiveness for standalone and coupled elastodynamic problems \cite{nguyen2011, duSayas2020, hungria2019, meddahi2023hp, fu2019, meddahi2025, meddahi2025b, meddahi2025c}. The present work extends the HDG framework proposed in \cite{meddahi2025b} for poroelasticity to the more complex case of thermo-poroelasticity. To our knowledge, no HDG strategy for thermo-poroelasticity has been developed previously.

Our point of departure is the linear, fully dynamic thermo-poroelasticity model considered in \cite{carcione2019,  bonetti2023numerical}, which couples solid mechanics, Darcy flow, and finite-speed heat transport. In this model, heat conduction follows the Maxwell-Cattaneo law, eliminating the non-physical infinite propagation velocity inherent in classical Fourier diffusion \cite{Biot1956b}. In contrast to \cite{carcione2019, bonetti2023numerical}, where the problem is formulated as a second-order PDE in time, we formulate the governing equations as a first-order hyperbolic system. This reformulation facilitates the application of semigroup theory for establishing well-posedness of the continuous problem. It is worth noting that \cite{SANTOS2021124907} establishes well-posedness through energy estimates and a classical Galerkin procedure, in contrast to our semigroup-based analysis.

Following \cite{meddahi2025}, we develop an energy-consistent HDG discretization on general simplicial meshes and present an $hp$-convergence analysis. Comprehensive numerical experiments validate the theoretical convergence rates and demonstrate the method's effectiveness for wave propagation phenomena in heterogeneous media, including comparisons with purely poroelastic models that highlight the significance of thermal coupling effects.

The remainder of the paper is organized as follows. Section~\ref{sec:model} presents the governing thermo-poroelastic equations. Section~\ref{sec:wellposedness} recasts the problem as an abstract first-order multi-field evolution equation and establishes existence, uniqueness, and continuous energy bounds by verifying the Lumer-Phillips conditions for the associated operator. Section~\ref{sec:semi-discrete} details the HDG discretization and establishes well-posedness for the resulting semi-discrete problem. Section~\ref{sec:convergence} derives a priori error bounds demonstrating optimal $hp$-convergence rates. Section~\ref{sec:numresults} presents numerical experiments that validate the theoretical predictions. Concluding remarks and perspectives for future research appear in Section~\ref{sec:conclusions}.

Throughout the rest of this paper, we shall use the letter $C$  to denote generic positive constants independent of the mesh size $h$ and the polynomial degree $k$. These constants may represent different values at different occurrences. Moreover, given any positive expressions $X$ and $Y$ depending on $h$ and $k$, the notation $X \,\lesssim\, Y$  means that $X \,\le\, C\, Y$.

\section{The Dynamic Linear Thermo-Poroelastic Model}\label{sec:model}

This section presents a dynamic linear thermo-poroelastic model describing the coupled mechanical, thermal, and hydraulic behavior of a fluid-saturated porous medium subjected to time-varying loads and temperature changes, cf. \cite{carcione2019, bonetti2023numerical}. 

Let $\Omega\subset\bbR^{d}$, $d\in\{2,3\}$, be a bounded Lipschitz domain with boundary $\Gamma$ and outward unit normal $\bn$. Let $T>0$ denote a fixed final time. Throughout, an overdot denotes the partial derivative with respect to time, i.e., $\dot v:=\partial_t v$. The model's behavior is characterized by the primary variables
\[
	(\bu,\bq,\br,\bsig,p,\theta):
	\Omega\times(0,T]\longrightarrow
	\bbR^{d}\times\bbR^{d}\times\bbR^{d}\times\bbS\times\bbR\times\bbR,
\]
where $\bu$ represents the velocity of the solid matrix, $\bq$ the Darcy  velocity of the fluid phase, $\bsig$ the effective Cauchy stress tensor, $p$ the pore pressure of the fluid, $\theta$ the absolute temperature, and $\br$ the heat flux. 

With these variables, the linear first-order thermo-poroelastic system reads:
\begin{subequations}\label{eq:TP}
	\begin{align}
		\rho\,\dot{\bu} + \rho_f\,\dot{\bq}
		- \bdiv \bigl(\bsig- (\alpha p + \beta\theta) I_d\bigr)
		 & = \bF_s,
		 &              & \text{(solid momentum)}\label{eq:TP-a}
		\\
		\rho_f\,\dot{\bu} + \rho_w\,\dot{\bq}
		+ \frac{\eta}{\kappa}\,\bq + \nabla p
		 & = \bF_f,
		 &              & \text{(fluid momentum)}\label{eq:TP-b}
		\\
		c_0\,\dot p - b_0\,\dot\theta + \alpha\,\div\bu + \div\bq
		 & = 0,
		 &              & \text{(mass balance)}\label{eq:TP-c}
		\\
		a_0\,\dot\theta - b_0\,\dot p + \beta\,\div\bu + \div\br
		 & = g,
		 &              & \text{(energy balance)}\label{eq:TP-d}
		\\
		\dot{\bsig} - \cC\,\beps(\bu)
		 & = \mathbf 0,
		 &              & \text{(elastic law)}\label{eq:TP-e}
		\\
		\frac{\tau}{\chi}\,\dot{\br} + \frac1\chi\,\br + \nabla\theta
		 & = \mathbf 0,
		 &              & \text{(Cattaneo law)}\label{eq:TP-f}
	\end{align}
\end{subequations}
Here, $I_d \in \mathbb{R}^{d \times d}$ is the identity in the space of real matrices $\mathbb{R}^{d\times d}$ and $\beps(\bu):= \frac{1}{2}\left[\bnabla\bu+(\bnabla\bu)^{\mt}\right] : \Omega \to \mathbb{R}^{d\times d}_{\text{sym}}$ is the linearized strain tensor, where $\mathbb{R}^{d \times d}_{\text{sym}} \coloneq \{\btau \in \mathbb{R}^{d \times d} \mid \btau = \btau^{\mt}\}$. Additionally, $\bdiv\bsig:=\bigl(\sum_{j=1}^{d}\partial_{x_j}\sigma_{ij}\bigr)_{i=1}^{d}$ is the divergence operator, and $\bF_s$, $\bF_f$, and $g$ represent body forces, fluid supply, and a heat source/sink term, respectively. The material parameters appearing in the equations are detailed in Table~\ref{tab:parameters}.

\begin{table}[ht]
    \centering
    \caption{Material parameters of the thermo-poroelastic system}\label{tab:parameters}
    \begin{tblr}{
            colspec={X[2,l] X[4,l]},
            row{1} = {font=\bfseries},
            hline{1,Z} = {1pt},
            hline{2} = {0.5pt},
        }
        Parameter                                  & Description                                                  \\
        $\rho_s$, $\rho_f$                         & Solid and fluid densities                                    \\
        $\phi\in(0,1)$                             & Porosity                                                     \\
        $\rho = \phi\rho_f+(1-\phi)\rho_s$         & Composite density of the mixture                             \\
        $\nu>1$, $\rho_w = \frac{\nu}{\phi}\rho_f$ & Tortuosity and added-mass coefficient                        \\
        $\eta$, $\kappa$                           & Fluid dynamic viscosity and intrinsic permeability           \\
        $\alpha\in(0,1]$                           & Biot-Willis coupling coefficient                             \\
        $c_0$                                      & Specific storage                                             \\
        $a_0$, $b_0$                               & Effective heat capacity and thermo-dilatation                \\
        $\beta$                                    & Thermo-stress coupling coefficient                           \\
        $\chi$, $\tau>0$                           & Fourier conductivity and Maxwell-Cattaneo relaxation time    \\
        $\mathcal{C}$                              & Fourth-order, symmetric, positive definite elasticity tensor \\
    \end{tblr}
\end{table}

The system is supplemented with initial conditions:
\begin{equation}\label{IC1}
	\bu(0)=\bu^0,\quad
	\bq(0)=\bq^0,\quad
	\br(0)=\br^0,\quad
	\bsig(0)=\bsig^0,\quad
	p(0)=p^0,\quad
	\theta(0)=\theta^0
	\qquad\text{in }\Omega,
\end{equation}
and boundary conditions defined on complementary parts of the boundary $\partial\Omega$. Specifically, $\Gamma=\Gamma_D^S\cup\Gamma_N^S
	=\Gamma_D^F\cup\Gamma_N^F
	=\Gamma_D^H\cup\Gamma_N^H$, where $\Gamma^\star_D$ denotes Dirichlet boundaries and $\Gamma^\star_N$ denotes Neumann boundaries for $\star \in \set{S,F,H}$. We assume $\text{meas}(\Gamma_D^S)$, $\text{meas}(\Gamma_D^F)$, $\text{meas}(\Gamma_D^H)>0$ to avoid rigid body motions or ill-posedness:
\begin{align}\label{BC}
	\begin{aligned}
		\bu                                   & = \mathbf 0
		                                        &             & \qquad\text{on }\Gamma_D^S\times(0,T], &
		(\bsig-\alpha p I_d-\beta\theta I_d)\bn & = \mathbf 0
		                                        &             & \qquad\text{on }\Gamma_N^S\times(0,T],
		\\
		p                                       & = 0
		                                        &             & \qquad\text{on }\Gamma_D^F\times(0,T], &
		\bq\!\cdot\!\bn                       & = 0
		                                        &             & \qquad\text{on }\Gamma_N^F\times(0,T],
		\\
		\theta                                  & = 0
		                                        &             & \qquad\text{on }\Gamma_D^H\times(0,T], &
		\br\!\cdot\!\bn                       & = 0
		                                        &             & \qquad\text{on }\Gamma_N^H\times(0,T].
	\end{aligned}
\end{align}
The homogeneous boundary conditions presented here are for simplicity. The term $(\bsig-\alpha p I_d-\beta\theta I_d)\bn$ represents the total traction on the solid phase boundary.

\begin{remark}
	Several key properties of system~\eqref{eq:TP} are worth noting. First, it constitutes a linear, first-order-in-time system that admits a non-increasing energy functional under appropriate positivity assumptions on the material parameters (detailed below). This energy structure is fundamental for establishing well-posedness in the classical $L^2$-framework. Second, the system reduces to Biot's poroelasticity when thermal effects are neglected, and to linear thermoelasticity with finite-speed heat conduction when poroelastic effects are absent. Third, the first-order formulation is particularly well-suited for energy-conserving high-order time integrators and hybridizable discontinuous Galerkin discretizations.
\end{remark}

The elasticity tensor $\cC$ is assumed to be symmetric and positive definite, meaning there exist constants $c^+ > c^- > 0$ such that
\begin{equation}\label{CD}
	c^- \boldsymbol{\zeta}:\boldsymbol{\zeta} \leq \cC \boldsymbol{\zeta} :\boldsymbol{\zeta} \leq c^+ \boldsymbol{\zeta}:\boldsymbol{\zeta}
	\quad \text{for all } \boldsymbol{\zeta}\in \mathbb{R}^{d\times d}_{\text{sym}}.
\end{equation}

For subsequent analysis, it is convenient to define the following matrices based on the material parameters:
\begin{equation}
	R := \begin{pmatrix}
		\rho   & \rho_f & 0                 \\
		\rho_f & \rho_w & 0                 \\
		0      & 0      & \frac{\tau}{\chi}
	\end{pmatrix},
	\quad \text{and} \quad
	Q := \begin{pmatrix}
		c_0  & -b_0 \\
		-b_0 & a_0
	\end{pmatrix}.
\end{equation}
The matrix $R$ groups coefficients of the time derivatives of velocities  and $Q$ relates to storage and coupling in the mass and energy balance equations.

Under the physical assumptions that the tortuosity $\nu > 1$ (which implies $\rho\rho_w - \rho_f^2 >0$ )  and that the matrix $Q$ is positive definite (i.e. $a_0>0$, $c_0>0$, and   $a_0c_0 - b_0^2 > 0$)  (ensuring positive definiteness of $R$ and $Q$), we can establish the following coercivity and boundedness inequalities for $R$ and $Q$:  
\begin{align}
	\rho^- \boldsymbol{\eta}:\boldsymbol{\eta} & \leq \boldsymbol{\eta}R :\boldsymbol{\eta} \leq \rho^+ \boldsymbol{\eta}:\boldsymbol{\eta}
	                                           &                                                                                            & \text{for all } \boldsymbol{\eta}\in \mathbb{R}^{d\times 3}, \label{eq:R_bounds} \\
	c^- \boldsymbol{w} \cdot \boldsymbol{w}    & \leq Q\boldsymbol{w} \cdot \boldsymbol{w} \leq c^+ \boldsymbol{w}\cdot \boldsymbol{w}
	                                           &                                                                                            & \text{for all } \boldsymbol{w}\in \mathbb{R}^2, \label{eq:Q_bounds}
\end{align}
where $\rho^+ > \rho^- > 0$ and $c^+ > c^- > 0$ are constants depending on the material parameters.

\section{Well-posedness of the model problem}\label{sec:wellposedness}

In this section, we establish the well-posedness of the model problem \eqref{eq:TP} with boundary conditions \eqref{BC} and initial conditions \eqref{IC1}, employing the theory of strongly continuous semigroups \cite{pazy}.

\subsection{Functional setting}

To facilitate the analysis of the coupled system \eqref{eq:TP-a}-\eqref{eq:TP-f}, we introduce a compact two-block notation, grouping the velocity components and heat flux in one block and the stress tensor together with the scalar variables into the other. To this end, for vectors $\boldsymbol{a}, \boldsymbol{b}, \boldsymbol{c} \in \mathbb{R}^d$, we define $\stack*{\boldsymbol{a} \mid \boldsymbol{b}\mid \boldsymbol{c}} \in \mathbb{R}^{d \times 3}$ as the matrix formed by the column-wise concatenation of $\boldsymbol{a}$, $\boldsymbol{b}$, and $\boldsymbol{c}$.  Similarly, for scalars $a$ and $b$, $\stack*{a \sep b}\in \mathbb{R}^2$ denotes the column vector with components $a$ and $b$.

To treat the velocity components and the heat flow in a unified manner, we define $\underline{\bu} \coloneq \stack*{\bu \mid \bq \mid \br }: \Omega\times [0,T] \to \mathbb{R}^{d \times 3}$. Using this notation, we can rewrite equations \eqref{eq:TP-a}, \eqref{eq:TP-b}, and \eqref{eq:TP-f} in the compact form:
\begin{equation}\label{compact1}
 \dot{\underline{\bu}}R +
 \stack*{ - \bdiv (\bsig - (\alpha p + \beta \theta) I_d) \mid \nabla p \mid \nabla \theta } +  \mathfrak I \underline{\bu} 
 = \underline{\bF},
\end{equation}
where $\underline{\bF} \coloneq \stack*{ \bF_s \mid \bF_f \mid \mathbf{0} }$ and $\mathfrak I$ is the linear operator given by $\mathfrak I \underline{\bu} \coloneq \stack*{ \mathbf{0} \mid \tfrac{\eta}{\kappa}\bq \mid \tfrac{1}{\chi}\br }$.  Analogously, equations \eqref{eq:TP-c}, \eqref{eq:TP-d}, and \eqref{eq:TP-e} can be expressed as:
\begin{subequations}\label{eq:TPvec}
 \begin{align}
 	Q \stack*{\dot{p} \sep \dot{\theta}}  + \stack*{ \alpha \div \bu + \div \bq
 		               \sep
 		               \beta \div \bu + \div \br }
 	                                & =  \stack*{0 \sep g},
 	\\
 	\dot{\bsig} - \cC\,\beps(\bu) & = \mathbf 0.
 \end{align}
 \end{subequations}

Finally, we introduce the coupled state variable $\boldsymbol{\Theta} := (\bsig, \stack*{p \sep \theta})$ and let $\boldsymbol{\Phi} := (\btau, \stack*{z \sep \varphi})$ be the associated test function. This grouping is motivated by the structure of the divergence term in the momentum equation, $\bdiv\bigl(\bsig - (\alpha p + \beta\theta) I_d\bigr)$, which intrinsically couples these variables and influences the regularity requirements imposed by the function space $\cX_2$ defined later. 

To establish a suitable framework for formulating the first-order system in time \eqref{compact1}-\eqref{eq:TPvec}, we begin by defining the pivot Hilbert spaces $\mathcal{H}_1\coloneq L^2(\Omega, \mathbb R^{d\times 3})$ and $\mathcal{H}_2 \coloneq  L^2(\Omega,\mathbb{R}^{d\times d}_{\text{sym}}) \times   [L^2(\Omega)]^2$. We endow $\mathcal{H}_1$ with the inner product $\inner{ \underline{\bu}, \underline{\bv} }_{\mathcal{H}_1} \coloneq  \inner{\underline{\bu} R, \underline{\bv}}_{\Omega}$. According to \eqref{eq:R_bounds}, the corresponding norm $\norm{\underline{\bv}}^2_{\mathcal{H}_1} \coloneq  \inner{ \underline{\bv}, \underline{\bv} }_{\mathcal{H}_1}$  satisfies
\begin{equation}\label{bound:rho}
	\rho^-  \norm{\underline{\bv}}^2_{0,\Omega}  \leq 	\norm{\underline{\bv}}^2_{\mathcal{H}_1} \leq \rho^+  \norm{\underline{\bv}}^2_{0,\Omega}  \quad \forall \underline{\bv}  =  \left[\bv \mid \bw \mid \bs \right] \in \mathcal{H}_1.
\end{equation}

Additionally, we define $\mathcal{A} \coloneq  \mathcal{C}^{-1}$, and equip
$\mathcal{H}_2$ with the inner product
\[
    \inner{ \boldsymbol{\Theta}, \boldsymbol{\Phi} }_{\mathcal{H}_2} \coloneq   \inner{\mathcal{A}\bsig, \btau}_\Omega  + \inner*{Q \stack*{p \sep \theta}, \stack*{z \sep \varphi} }_\Omega,
\]
where $\boldsymbol{\Theta} = (\bsig, \stack*{p \sep \theta})$ and $\boldsymbol{\Phi}=(\btau,\stack*{z \sep \varphi})$ are elements of $\mathcal{H}_2$. The associated norm is $\norm{\boldsymbol{\Phi} }^2_{\mathcal{H}_2}  \coloneq   \inner{\mathcal{A}\btau, \btau}_\Omega  + \inner{Q \stack*{z \sep \varphi}, \stack*{z \sep \varphi} }_\Omega$. By \eqref{CD} and \eqref{eq:Q_bounds}, there exist  constants $a^+ > a^- >0$ such that
\begin{equation}\label{normH}
	a^- \left( \norm{\btau}^2_{0, \Omega}  + \norm{z}^2_{0, \Omega} + \norm{\varphi}^2_{0, \Omega}  \right) \leq \norm{\boldsymbol{\Phi} }^2_{\mathcal{H}_2} \leq a^+ \left( \norm{\btau}^2_{0, \Omega}  + \norm{z}^2_{0, \Omega} + \norm{\varphi}^2_{0, \Omega}  \right)   \quad \forall  \boldsymbol{\Phi}    \in \mathcal{H}_2.
\end{equation}

To incorporate the essential boundary conditions specified in \eqref{BC} within the energy space associated with the fluid velocity component and the heat flow, we introduce the closed subspaces
\[
	H_\star(\div, \Omega) := \set*{ \bv\in H(\div, \Omega); \quad \dual{\bv\cdot \bn,w}_{\Gamma}= 0 	\quad \forall w\in H^1_\star(\Omega) },
	\quad \star \in \set{F,H}
\]
consisting of vector fields in $H(\div, \Omega):= \set{\bv \in L^2(\Omega,\bbR^d);\ \div \bv \in L^2(\Omega)}$ with a free normal component on $\Gamma^\star_N$, $\star \in \set{F,H}$. Here,  $H^1_\star(\Omega):=\set{w\in H^1(\Omega);\ w|_{\Gamma^\star_D} = 0}$, $\star \in \set{F,H}$, and $\dual{\cdot, \cdot}_\Gamma$ represent the duality pairing between $H^{\sfrac12}(\Gamma)$ and $H^{-\sfrac12}(\Gamma)$.

Similarly, we let $H^1_S(\Omega,\bbR^d):=\set{\bv\in H^1(\Omega,\bbR^d);\ \bv|_{\Gamma^S_D} = \mathbf 0}$ and define
\[
	H_S(\bdiv, \Omega, \mathbb{R}^{d\times d}_{\text{sym}}) := \set{ \btau\in H(\bdiv, \Omega, \mathbb{R}^{d\times d}_{\text{sym}}); \quad
	\left\langle \btau\bn,\bu \right\rangle_{\Gamma}= 0
	\quad \forall\bu\in H^1_S(\Omega,\bbR^d) },
\]
as the closed subspace of 
$ 
H(\bdiv, \Omega, \mathbb{R}^{d\times d}_{\text{sym}}) := \set{\btau \in L^2(\Omega, \mathbb{R}^{d\times d}_{\text{sym}}); \ \bdiv \btau \in L^2(\Omega, \bbR^d) }
$ 
satisfying a stress-free boundary condition on $\Gamma^S_N$.

We define the unbounded operator $\mathfrak{L} :\ \cX_1\subset \mathcal{H}_1 \to \mathcal{H}_2$ with domain $\cX_1 \coloneq H_S^1(\Omega,\bbR^d) \times H_F(\div,\Omega) \times H_H(\div,\Omega)$ by
\[
	\mathfrak{L} \underline{\bu}  \coloneq  \left( \cC \beps(\bu), -Q^{-1} \stack*{ \alpha \div \bu + \div \bq
		\sep
		\beta \div \bu + \div \br}
	\right).
\]

\begin{lemma}\label{lem:L}
The operator $\mathfrak{L}:\ \mathcal{X}_1 \subset \mathcal{H}_1  \to \mathcal{H}_2$ is densely defined and closed.
\end{lemma}
\begin{proof}
The density of $\mathcal{X}_1$ in $\mathcal{H}_1$ follows from the well-known density of $H^1(\Omega,\bbR^d) \subset H(\div, \Omega)$ in $L^2(\Omega,\bbR^d)$. To show that the graph of $\mathfrak{L}$ is closed, assume that $\set{\underline{\bu}_n}_n \subset \mathcal{X}_1$ is such that  $\underline{\bu}_n \to \underline{\bu}$ in $\mathcal{H}_1$ and $\mathfrak{L}\underline{\bu}_n \to \underline{\bz}$ in $\mathcal{H}_2$. It follows that $\set{\beps(\bu_{n})}_n$ is a Cauchy sequence in $L^2(\Omega,\bbS)$, and Korn's inequality guarantees that $\bu_{n} \to \bu$ in $H_S^1(\Omega,\bbR^d)$. We also have that $\set{\alpha\div\bu_{n} + \div\bq_{n}}_n$ and $\set{\beta\div\bu_{n} + \div\br_{n}}_n$ are  Cauchy sequences in $L^2(\Omega)$. Consequently, $\bq_{n} \to \bq$ in $H_F(\div, \Omega)$ and $\br_{n} \to \br$ in $H_H(\div, \Omega)$. Uniqueness of the limit in $\mathcal{H}_1$ establishes that $\underline{\bz} = \mathfrak{L}\underline{\bu}$, completing the proof.
\end{proof}

A direct calculation shows that the operator $\mathfrak{L}^H :\ \mathcal{X}_2\subset \cH_2 \to \mathcal{H}_1$ given by 
\[
	\mathfrak{L}^H \boldsymbol{\Theta} \coloneq \stack*{-\bdiv(\bsig - (\alpha p + \beta \theta) I_d ) \mid \nabla p \mid \nabla \theta } R^{-1},
\]
with domain 
\[
		\cX_2 \coloneq  \Big\{ \boldsymbol{\Phi}= (\btau, \stack{z \sep \varphi}) \in L^2(\Omega, \bbS)\times [H^1_F(\Omega)\times H^1_H(\Omega)];\
		\bdiv\left(\btau - (\alpha z + \beta  \varphi) I_d \right) \in L^2(\Omega,\bbR^d) \Big\},
\]
is the adjoint of $\mathfrak{L}$. It is characterized  by
\[
	\inner{\mathfrak{L} \underline{\bv} , \boldsymbol{\Phi}}_{\cH_2} = \inner{\underline{\bv}, \mathfrak{L}^H \boldsymbol{\Phi}}_{\cH_1} = \quad \forall \underline{\bv} \in \cX_1,  \ \boldsymbol{\Phi}  \in \cX_2. 
\]
It follows from Lemma \ref{lem:L} that $\mathfrak{L}^H :\ \mathcal{X}_2\subset \cH_2 \to \mathcal{H}_1$ is also a closed operator. Hence, the linear spaces $\mathcal{X}_1$ and $\mathcal{X}_2$ are Hilbert spaces when equipped with the inner products
\begin{align}
		 \inner{\underline{\bu} , \underline{\bv} }_{\cX_1} &\coloneq  \inner{\underline{\bu} , \underline{\bv} }_{\cH_1} +  \inner*{\mathfrak{L}\underline{\bu} , \mathfrak{L}\underline{\bv} }_{\cH_2} \label{eq:inner1}
		 \\
		 \inner{\boldsymbol{\Theta}, \boldsymbol{\Phi} }_{\cX_2} &\coloneq  \inner{\boldsymbol{\Theta}, \boldsymbol{\Phi} }_{\cH_2} + \inner{\mathfrak{L}^H\boldsymbol{\Theta}, \mathfrak{L}^H\boldsymbol{\Phi} }_{\cH_1}, \label{eq:inner2}
\end{align}
respectively.

\subsection[Application of C0-semigroup theory]{Application of $C_0$-semigroup theory}

The coupled system of partial differential equations \eqref{eq:TP} with the initial conditions \eqref{IC1} and boundary conditions \eqref{BC} can be reformulated as a first-order evolution problem in a Hilbert space. This allows us to utilize the  theory of strongly continuous semigroups for existence and uniqueness analysis.

The system \eqref{compact1}--\eqref{eq:TPvec} and the initial conditions \eqref{IC1} can be written compactly as
\begin{equation}\label{IVP}
   \frac{\mathrm{d}}{\mathrm{d} t}
   \begin{pmatrix}\underline \bu \\ \boldsymbol\Theta\end{pmatrix}
   +\mathbb A
   \begin{pmatrix}\underline \bu \\ \boldsymbol\Theta\end{pmatrix}
   \;=\;
   \begin{pmatrix}\underline{\bF}R^{-1} \\ \boldsymbol G\end{pmatrix}, 
   \quad
   \begin{pmatrix}\underline \bu(0) \\ \boldsymbol\Theta(0)\end{pmatrix}
   =\begin{pmatrix}\underline{\bu}^0 \\ \boldsymbol{\Theta}^0\end{pmatrix}
\end{equation}
where $\boldsymbol G \coloneq (\mathbf 0, Q^{-1}\stack*{0 \sep g})$, $\underline{\bu}^0 \coloneq  \stack*{\bu^0 \mid \bq^0 \mid \br^0 }$ and $\boldsymbol{\Theta}^0 \coloneq (\bsig^0, \stack*{p^0 \sep \theta^0})$. The infinitesimal generator $\mathbb{A}$ of the strongly continuous semigroup on $\cH_1 \times \cH_2$ associated with the evolution system \eqref{IVP} is given by
\[
	\mathbb{A} \begin{pmatrix} \underline{\bu} \\ \boldsymbol{\Theta} \end{pmatrix} \coloneq \begin{pmatrix} \mathbf 0 & \mathfrak{L}^H \\ -\mathfrak{L} & \mathbf 0 \end{pmatrix} \begin{pmatrix} \underline{\bu} \\ \boldsymbol{\Theta} \end{pmatrix} + \begin{pmatrix} \mathfrak I \underline{\bu}  R^{-1} \\ \mathbf 0 \end{pmatrix}.
\]

\begin{theorem}\label{thm:Hille-Yosida}
For all $\underline{\bF}  \in \mathcal{C}^1_{[0,T]}(L^2(\Omega, \mathbb{R}^{d\times 3}))$, $g \in \mathcal{C}^1_{[0,T]}(L^2(\Omega))$, $\underline{\bu}^0 \in \mathcal{X}_1$, and $\boldsymbol{\Theta}^0 \in \mathcal{X}_2$, there exist unique $\underline{\bu} \in \mathcal{C}^1_{[0,T]}(\mathcal{H}_1) \cap \mathcal{C}^0_{[0,T]}(\mathcal{X}_1)$ and $\boldsymbol{\Theta} \in \mathcal{C}^1_{[0,T]}(\mathcal{H}_2) \cap \mathcal{C}^0_{[0,T]}(\mathcal{X}_2)$ solutions to the initial boundary value problem \eqref{IVP}. Moreover, there exists a constant $C>0$ such that  
\begin{equation}\label{eq:stab}
\max_{[0,T]}\norm{\underline{\bu}(t)}_{\mathcal{H}_1} + \max_{[0,T]}\norm{\boldsymbol{\Theta}(t) }_{\mathcal{H}_2} \leq C \left( \max_{[0,T]}\norm{\underline{\bF}(t)}_{0,\Omega} + \max_{[0,T]}\norm{g(t)}_{0,\Omega}   \right) +   \norm{\underline{\bu}^0}_{\mathcal{H}_1} + \norm{\boldsymbol{\Theta}^0} _{\mathcal{H}_2}.
\end{equation}
\end{theorem}
\begin{proof}
By definition, we have that
\[
\inner*{\mathbb A
   \begin{pmatrix}\underline \bu\\\boldsymbol\Theta\end{pmatrix}, 
   \begin{pmatrix}\underline \bu\\\boldsymbol\Theta\end{pmatrix}}_{\cH_1 \times \cH_2} = \inner*{ \mathfrak I \underline{\bu}, \underline{\bu}
   }_{\Omega} \geq 0 \quad \forall \begin{pmatrix}\underline \bu\\\boldsymbol\Theta\end{pmatrix} \in \cX_1 \times \cX_2,
\]
which proves that $\mathbb A$ is monotone.

It remains to show that the operator $\mathbb I + \mathbb A: \mathcal{X}_1 \times \mathcal{X}_2 \to  \mathcal{H}_1 \times \mathcal{H}_2$ is surjective, where $\mathbb I$ represents the identity operator in $\mathcal{X}_1 \times \mathcal{X}_2$. Given an arbitrary  $\left( \underline{\bF} ,\boldsymbol G  \right)\in  \mathcal{H}_1 \times \mathcal{H}_2$, we seek $\left( \underline{\bu}^* ,\boldsymbol\Theta^*  \right)\in  \mathcal{X}_1 \times \mathcal{X}_2$ satisfying 
\[
\left( \mathbb I + \mathbb A \right) \begin{pmatrix}\underline \bu^*\\\boldsymbol\Theta^*\end{pmatrix} = \begin{pmatrix}\underline \bF\\\boldsymbol G\end{pmatrix}.
\]
Writing out the components,
\begin{equation}\label{system}
	   \underline \bu^*  + \mathfrak L^H \boldsymbol\Theta^* + \mathfrak I \underline{\bu}^* R^{-1}
   =\underline \bF,
   \quad
   \boldsymbol\Theta^* - \mathfrak L \underline \bu^*
   =\boldsymbol G,
\end{equation}
and substituting \(\boldsymbol\Theta^* = \boldsymbol G + \mathfrak L \underline \bu^*\) into the first equation, we obtain the variational problem: find \(\underline \bu^*\in\cX_1\) such that
\begin{equation}\label{eq:var2}
   \inner*{\underline \bu^*,\underline \bv}_{\cX_1} + \inner*{ \mathfrak I \underline{\bu}^* R^{-1},\underline \bv}_{\cH_1}
   =
   \inner*{\underline \bF,\underline \bv}_{\cH_1}
   -
   \inner*{\boldsymbol G, \mathfrak L\underline \bv}_{\cH_1}
   \quad
   \forall\,\underline v\in\cX_1.
\end{equation}
By the Lax--Milgram lemma, there exists a unique solution \(\underline \bu^* \in \cX_1 \) to this problem. 

To ensure that $\boldsymbol\Theta^* = \boldsymbol G + \mathfrak L \underline \bu^* \in \mathcal{H}_2$ actually lies in $\cX_2$,  we employ a dual procedure to the first step and substitute $\underline \bu^*=   - \mathfrak L^H \boldsymbol\Theta^* - \mathfrak I \underline{\bu}^* R^{-1}
   +\underline \bF$ into the second equation of \eqref{system} to deduce that   $\boldsymbol\Theta^*\in\cX_2$ is the unique solution of the variational problem
   \begin{equation}\label{eq:var3}
   \inner*{\boldsymbol\Theta^* ,\boldsymbol\Phi}_{\cX_2} = \inner*{\mathfrak I \underline{\bu}^* R^{-1}
   -\underline \bF, \mathfrak L^H  \boldsymbol\Phi}_{\cH_1} + \inner*{\boldsymbol G, \boldsymbol\Phi}_{\cH_2}\quad \forall  \boldsymbol\Phi \in \cX_2.
\end{equation}

We have shown that solving problems \eqref{eq:var2} and \eqref{eq:var3} successively provides an inverse image $(\underline{\boldsymbol{u}}^*, \boldsymbol{\Theta}^*) \in \mathcal{X}_1 \times \mathcal{X}_2$ of $(\underline{\boldsymbol{F}}, \boldsymbol{G}) \in \mathcal{H}_1 \times \mathcal{H}_2$ under $\mathbb{I} + \mathbb{A}$.

Since $\mathbb{A}$ is maximal monotone, the result follows from the Lumer-Phillips variant of the Hille-Yosida theorem (cf. \cite[Theorem 76.7]{ErnBook2021III}).
\end{proof}

We define the bounded bilinear form $B: \cX_1 \times \cX_2 \to \bbR$ as
\[
B( \underline{\bv} , \boldsymbol{\Phi}) \coloneq   \inner{\bdiv(\btau - (\alpha z + \beta  \varphi) I_d, \bv}_\Omega  -  \inner{ \nabla z,\bq}_\Omega -  \inner{ \nabla \varphi ,\br}_\Omega.
\]
In the following sections, we will develop a finite element discretization method based on the following weak form of \eqref{IVP}: Find $\underline{\bu}= \stack*{\bu \mid \bq \mid \br} \in \mathcal{C}_{[0,T]}^1(\cH_1) \cap \mathcal{C}^0_{[0,T]}(\cX_1)$ and  $\boldsymbol{\Theta} = (\bsig, \stack{p \sep \theta}) \in \mathcal{C}^1_{[0,T]}(\cH_2) \cap \mathcal{C}^0_{[0,T]}(\cX_2)$ satisfying 
\begin{align}\label{eq:weakform}
	\begin{split}
		\inner*{\dot{\underline{\bu}}  , \underline{\bv} }_{\cH_1} + \inner*{\mathfrak I \underline{\bu} ,\underline{\bv} }_\Omega  -B( \underline{\bv} , \boldsymbol{\Theta})
		&= \inner*{\underline{\bF}, \underline{\bv} }_\Omega 
		\\
	\inner{\dot{\boldsymbol{\Theta}}, \boldsymbol{\Phi} }_{\cH_2} 
	+  B( \underline{\bu} , \boldsymbol{\Phi})
	&= \inner*{g, \varphi}_\Omega,
	\end{split}
\end{align}
for all $\underline{\bv} = \stack*{\bv \mid \bw \mid \bs}\in \cX_1$ and all $\boldsymbol{\Phi} = (\btau, \stack{z \sep \varphi})\in \cX_2$, and subject to the initial conditions specified  in \eqref{IVP}.

\section{The semi-discrete problem}\label{sec:semi-discrete}

In this section, we present the HDG semi-discrete method for the poroelasticity problem \eqref{eq:weakform}. We establish the well-posedness of the resulting algebraic-differential equations and prove the method's consistency with the continuous formulation.

For the sake of simplicity, from now on we assume that $\Gamma^S_D = \Gamma^F_D = \Gamma^H_D = \Gamma$. As a result, the boundary conditions \eqref{BC} become 
\begin{equation}\label{newBC}
		\bu = \mathbf 0,  \quad  
	p = 0, \quad \text{and} 
	\quad \theta = 0 \quad \text{on $ \Gamma\times (0, T]$}.
\end{equation}   

\subsection{Broken Sobolev Spaces}

Let $\cT_h$ be a shape-regular partition of the computational domain closure $\bar{\Omega}$ into elements. These elements are tetrahedra or parallelepipeds when the spatial dimension $d=3$, and triangles or quadrilaterals when $d=2$. The partition $\cT_h$ is permitted to have hanging nodes. For each element $K \in \cT_h$, we denote its diameter by $h_K$. The mesh size parameter $h$ is defined as the maximum diameter among all elements, i.e., $h := \max_{K \in \cT_h} \{h_K\}$.

A closed subset $F \subset \overline{\Omega}$ is classified as an interior edge or face if it possesses a positive $(d-1)$-dimensional measure and can be expressed as the intersection of the closures of two distinct elements $K$ and $K'$, formally $F = \bar{K} \cap \bar{K}'$. Conversely, $F \subset \overline{\Omega}$ is a boundary edge or face if there exists an element $K \in \cT_h$ such that $F$ is an edge or face of $K$, and $F = \bar{K} \cap \partial\Omega$. We denote the set of all interior edges/faces by $\cF_h^0$ and the set of all boundary edges/faces by $\cF_h^\partial$. The union of these sets constitutes the set of all mesh faces, $\cF_h = \cF_h^0 \cup \cF_h^\partial$. For each face $F \in \cF_h$, $h_F$ represents its diameter. The mesh $\cT_h$ is assumed to be locally quasi-uniform with a constant $\gamma > 0$. This condition implies that for all $h$ and for every element $K \in \cT_h$, the following relationship holds:
\begin{equation}\label{reguT}
	h_F \leq h_K \leq \gamma h_F \quad \forall F \in \cF(K),
\end{equation}
where $\cF(K)$ denotes the set of all faces comprising the boundary of element $K \in \cT_h$. This local quasi-uniformity is a prerequisite for the proof of Lemma~\ref{maintool2}.

For any $s \geq 0$, the broken Sobolev space with respect to the partition $\cT_h$ of $\bar{\Omega}$ is defined as the space
\[
	H^s(\cT_h,\mathbb{R}^{m\times n}) :=
	\left\{\bv \in L^2(\Omega, \mathbb{R}^{m\times n}) \mid \bv|_K \in H^s(K, \mathbb{R}^{m\times n}) \quad \forall K \in \cT_h \right\}.
\]
Following standard conventions, we simplify notation where the codomain is $\mathbb{R}$, writing $H^s(\cT_h,\bbR) = H^s(\cT_h)$, and note that for $s=0$, $H^0(\cT_h,\mathbb{R}^{m\times n}) = L^2(\cT_h,\mathbb{R}^{m\times n})$. We introduce an inner product on $L^2(\cT_h, \mathbb{R}^{m\times n})$ as the sum of element-wise inner products:
\[
	\inner{\boldsymbol{\psi}, \boldsymbol{\varphi}}_{\cT_h} := \sum_{K\in \cT_h} \inner{\boldsymbol{\psi}, \boldsymbol{\varphi}}_{K} \quad \forall \boldsymbol{\psi}, \boldsymbol{\varphi} \in L^2(\cT_h, \mathbb{R}^{m\times n}).
\]
The corresponding norm is denoted by $\norm{\boldsymbol{\psi}}^2_{0,\cT_h} := \inner{\boldsymbol{\psi}, \boldsymbol{\psi}}_{\cT_h}$.
Furthermore, we define $\partial \cT_h := \set{\partial K \mid K \in \cT_h}$ as the collection of all element boundaries. The space $L^2(\partial \cT_h,\mathbb{R}^{m\times n})$ comprises $m \times n$ matrix-valued functions that are square-integrable on each $\partial K \in \partial \cT_h$. We define the following inner product and induced norm on this space:
\[
	\dual{\boldsymbol{\psi}, \boldsymbol{\varphi}}_{\partial \cT_h} := \sum_{K\in \cT_h} \dual{\boldsymbol{\psi},\boldsymbol{\varphi}}_{\partial K},
	\quad \text{and} \quad
	\norm{\boldsymbol{\varphi}}^2_{0, \partial \cT_h} := \dual{\boldsymbol{\varphi}, \boldsymbol{\varphi}}_{\partial \cT_h}
	\quad
	\forall \boldsymbol{\psi}, \boldsymbol{\varphi}\in L^2(\partial \cT_h,\mathbb{R}^{m\times n}),
\]
where the element boundary integral is given by $\dual{\boldsymbol{\psi}, \boldsymbol{\varphi}}_{\partial K} := \sum_{F\in \cF(K)} \int_F \boldsymbol{\psi}: \boldsymbol{\varphi}$.
Additionally, the space $L^2(\cF_h,\mathbb{R}^{m\times n})$ is equipped with the inner product:
\[
	(\boldsymbol{\psi}, \boldsymbol{\varphi})_{\cF_h} := \sum_{F\in \cF_h} \int_F\boldsymbol{\psi}: \boldsymbol{\varphi} \quad \forall \boldsymbol{\psi}, \boldsymbol{\varphi}\in L^2(\cF_h,\mathbb{R}^{m\times n}),
\]
with the associated norm denoted by $\norm*{\boldsymbol{\varphi}}^2_{0,\cF_h} := (\boldsymbol{\varphi},\boldsymbol{\varphi})_{\cF_h}$.

It is crucial to distinguish between functions defined on $L^2(\partial \cT_h,\mathbb{R}^{m\times n})$ versus $L^2(\cF_h,\mathbb{R}^{m\times n})$. Functions in $L^2(\partial \cT_h,\mathbb{R}^{m\times n})$ can possess two distinct values on each interior face $F$, corresponding to the trace from each adjacent element. In contrast, functions in $L^2(\cF_h,\mathbb{R}^{m\times n})$ are single-valued on each face $F$.

To formulate the HDG method, we need to introduce product spaces that accommodate both volume and trace variables. For $r > \frac{1}{2}$, we define:
\begin{align*}
	\begin{split}
			\mathcal{U} &\coloneq \left\{\ubv=[\bv \mid \bw \mid \bs] \mid \bv \in H^1(\cT_h,\mathbb{R}^d),\ \bw, \bs \in H(\div,\cT_h) \cap H^r(\cT_h,\mathbb{R}^d)\right\},
	\\
\hat{\mathcal{U}} &\coloneq \left\{\hat{\underline{\bv} } = [\hat{\bv} \mid \hat{\bw} \mid \hat{\bs}] \mid \hat{\bv} \in L^2(\cF^0_h,\mathbb{R}^{d}),\ \hat{\bw}, \hat{\bs} \in L^2(\cF_h,\mathbb{R}^{d}) \right\},
	\end{split}
\end{align*}
where $L^2(\cF^0_h,\mathbb{R}^{d}) \coloneq \left\{\boldsymbol{\phi}\in L^2(\cF_h,\mathbb{R}^{d}) \mid \boldsymbol{\phi}|_F = \mathbf{0},\ \forall F \in \cF_h^\partial\right\}$.
We endow the product space $\mathcal{U} \times \hat{\mathcal{U}}$ with the semi-norm defined for all $(\ubv, \hat{\ubv}) \in \mathcal{U} \times \hat{\mathcal{U}}$ by:
\begin{equation}\label{norm:sym}
	\abs*{(\ubv, \hat{\ubv})}^2_{\mathcal{U} \times \hat{\mathcal{U}}} = \norm{\beps( \bv)}^2_{0,\cT_h} + \norm{\div \bw}^2_{0,\cT_h} + \norm{\div \bs}^2_{0,\cT_h} + \norm{\frac{k+1}{h_\cF^{\frac{1}{2}}}(\ubv - \hat{\ubv})}^2_{0, \partial \cT_h},
\end{equation}
where $h_\cF \in \cP_0(\cF_h)$ is a piecewise constant function defined such that $h_\cF|_F := h_F$ for all $F \in \cF_h$.

\subsection{Piecewise-polynomial Spaces}

Hereafter, $\cP_\ell(D)$ denotes the space of polynomials. Specifically, if $D$ is a triangle or tetrahedron, $\cP_\ell(D)$ consists of polynomials of total degree at most $\ell \geq 0$. If $D$ is a quadrilateral or parallelepiped, $\cP_\ell(D)$ comprises polynomials of degree at most $\ell$ in each variable. The space of $\mathbb{R}^{m\times n}$-valued functions whose components belong to $\cP_\ell(D)$ is denoted by $\cP_\ell(D,\mathbb{R}^{m\times n})$. In particular, $\cP_\ell(D,\mathbb{R}^{d\times d}_{\text{sym}})$ refers to symmetric $d\times d$ matrices with components in $\cP_\ell(D)$. We introduce the space of piecewise-polynomial functions defined over the mesh partition $\cT_h$ as:
$$
	\cP_\ell(\cT_h) :=
	\set{ v\in L^2(\Omega): \ v|_K \in \cP_\ell(K),\ \forall K\in \cT_h }.
$$
Similarly, the space of piecewise-polynomial functions defined on the mesh skeleton $\cF_h$ is given by:
$$
	\cP_\ell(\cF_h) :=
	\set{ \hat v\in L^2(\cF_h): \ \hat v|_F \in \cP_\ell(F),\ \forall F\in \cF_h }.
$$
The subspace of $L^2(\cT_h, \mathbb{R}^{m\times n})$ with components in $\cP_\ell(\cT_h)$ is written $\cP_\ell(\cT_h, \mathbb{R}^{m\times n})$. Analogously, $\cP_\ell(\cF_h, \mathbb{R}^{m\times n})$ represents the subspace of $L^2(\cF_h, \mathbb{R}^{m\times n})$ with components in $\cP_\ell(\cF_h)$. Finally, we consider the space of piecewise-polynomial functions on the element boundaries:
$$
\cP_\ell(\partial \cT_h, \mathbb{R}^{m\times n}) := \set*{\phi\in L^2(\partial \cT_h, \mathbb{R}^{m\times n});\ \phi|_{\partial K}\in  \cP_\ell(\partial K,\mathbb{R}^{m\times n}),\ \forall K\in \cT_h},
$$
where $\cP_\ell(\partial K, \mathbb{R}^{m\times n})$ is defined as the product space $\prod_{F\in \cF(K)} \cP_\ell(F, \mathbb{R}^{m\times n})$.

We define $\bn\in \cP_0(\partial \cT_h, \mathbb{R}^d)$ as the piecewise constant unit normal vector field, where $\bn|_{\partial K}=\bn_K$ is the unit normal vector of $\partial K$ oriented outwards from $K$. It is important to note that if $F = K\cap K'$ is an interior face of $\cF_h$, then the normal vectors are oppositely oriented, i.e., $\bn_K = -\bn_{K'}$ on $F$. 

\subsection{The HDG Method}

For a given polynomial degree $k \geq 0$, we introduce the finite-dimensional subspaces $\mathcal{H}_{1,h}$ and $\mathcal{H}_{2,h}$ approximating $\mathcal{H}_1$ and $\mathcal{H}_2$, respectively:
$$
	\mathcal{H}_{1,h} \coloneq \cP_{k+1}(\cT_h,\bbR^{d\times 3})\quad \text{and} \quad \mathcal{H}_{2,h} \coloneq \cP_{k}(\cT_h,\mathbb{R}^{d\times d}_{\text{sym}}) \times [\cP_{k}(\cT_h)]^2.
$$
Additionally, for $k \geq 0$, we define the finite-dimensional subspace $\hat{\mathcal{H}}_{1,h}$ as:
$$
\hat{\mathcal{H}}_{1,h} \coloneq \set*{[\hat \bv \mid \hat \bw \mid \hat \bs] \mid \hat \bv \in \cP_{k+1}(\cF^0_h, \mathbb{R}^d),\ \hat \bw, \hat \bs \in \cP_{k+1}(\cF_h, \mathbb{R}^d)} \subset \hat{\mathcal{U}},
$$
where $\cP_{k+1}(\cF^0_h, \mathbb{R}^d) := \set{\boldsymbol{\phi} \in \cP_{k+1}(\cF_h, \mathbb{R}^d) \mid \boldsymbol{\phi}|_F = \mathbf 0,\ \forall F\in \cF_h^\partial}$.

We propose the following HDG spatial discretization method for problem \eqref{eq:weakform}: find $(\underline{\bu}_{h}, \underline{\hat{\bu}}_{h} ) \in \cC^1_{[0,T]}(\mathcal{H}_{1,h} \times \hat{\mathcal{H}}_{1,h})$ and $\boldsymbol{\Theta}_h = (\bsig_h, [p_h \sep \theta_h] ) \in \cC^1_{[0,T]}(\cH_{2,h})$ satisfying:
\begin{align}\label{sd}
	\begin{split}
		\inner*{\dot{\underline{\bu}}_{h}, \underline{\bv}}_{\cH_1} +
		\inner{\dot{\boldsymbol{\Theta}}_h, \boldsymbol{\Phi} }_{\cH_2} &+ \inner{\mathfrak I \underline{\bu}_{h},\underline{\bv} }_\Omega
		+ B_h(\boldsymbol{\Theta}_h, (\underline{\bv}, \underline{\hat{\bv}} ))
		- B_h(\boldsymbol{\Phi}, (\underline{\bu}_{h}, \underline{\hat{\bu}}_{h} ))
		\\
	  & \quad +\dual{ \tfrac{(k+1)^2}{h_\cF}(\underline{\bu}_{h} - \underline{\hat{\bu}}_{h} ),
	  \underline{\bv} - \underline{\hat{\bv}} }_{\partial \cT_h}
	  =  \inner{\underline{\bF}, \underline{\bv}}_\Omega + \inner{g, \varphi}_\Omega,
	\end{split}
\end{align}
for all $(\underline{\bv}, \underline{\hat{\bv}} ) \in \mathcal{H}_{1,h} \times \hat{\mathcal{H}}_{1,h}$ and $\boldsymbol{\Phi} = (\btau, [z \sep \varphi] ) \in \cH_{2,h}$. The bilinear form $B_h(\cdot,\cdot)$ is given by:
\begin{align*}
\begin{split}
	B_h( \boldsymbol{\Phi}, (\underline{\bv}, \underline{\hat{\bv}} )) &:=
	\inner*{\btau - (\alpha z + \beta \theta) I_d, \beps(\bv) }_{\cT_h} - \inner*{z, \div\bw}_{\cT_h} - \inner*{\varphi, \div\bs}_{\cT_h}
	\\
	& \qquad - \dual*{(\btau - (\alpha z + \beta \theta) I_d)\bn , \bv - \hat\bv }_{\partial \mathcal{T}_h} + \dual*{z\bn, \bw - \hat\bw }_{\partial \mathcal{T}_h} + \dual*{\varphi\bn, \bs - \hat\bs }_{\partial \mathcal{T}_h}.
\end{split}
\end{align*}
Problem \eqref{sd} is initialized with the following conditions:
\begin{align}\label{initial-R1-R2-h*c}
	\begin{split}
			\ubu_h(0) &= \Pi^{k+1}_\cT \ubu^0, \quad \hat{\ubu}_h(0) = \Pi^{k+1}_\cF (\ubu^0|_{\partial \mathcal{F}_h}),
	\\
	 \bsig_{h}(0)&= \Pi_\cT^k \bsig^0, \quad p_h(0) = \Pi_\cT^k p^0 \quad \text{and} \quad \theta_h(0) = \Pi_\cT^k \theta^0,
	\end{split}
\end{align}
where the $L^2$-projection operators $\Pi^{k}_\cT$ and $\Pi^{k+1}_\cF$ are defined in Appendix~\ref{appendix}.

The bilinear form $B_h(\cdot,\cdot)$ possesses the following boundedness property.
\begin{proposition}
	There exists a constant $C>0$ independent of $h$ and $k$ such that:
	\begin{equation}\label{Bhh}
		|B_h(\boldsymbol{\Phi}_h, (\underline{\bv}, \underline{\hat{\bv}} ))| \leq C \norm{\boldsymbol{\Phi}_h }_{\cH_2} \abs{ (\underline{\bv}, \underline{\hat{\bv}} )}_{\mathcal{U}\times \hat{\mathcal{U}}}\quad \text{for all}\ \boldsymbol{\Phi}_h \in \mathcal H_{2,h} \ \text{and} \ (\underline{\bv}, \underline{\hat{\bv}} ) \in \mathcal{U}\times \hat{\mathcal{U}}.
	\end{equation}
\end{proposition}
\begin{proof}
	We recall that if $\boldsymbol{\varphi} \in H^s(\cT_h, \bbR^{m\times n})$, with $s>\sfrac{1}{2}$, then $\boldsymbol{\varphi}|_{\partial \cT_h}\in L^2(\partial \cT_h,\bbR^{m\times n})$ is well-defined due to the trace theorem. Applying the Cauchy-Schwarz inequality and \eqref{normH}, we deduce that:
	\begin{equation}\label{Bh}
		|B_h(\boldsymbol{\Phi}, (\underline{\bv}, \underline{\hat{\bv}} ))| \lesssim ( \norm{\boldsymbol{\Phi} }^2_{\cH_2} + \norm{\tfrac{h_\cF^{\sfrac{1}{2}}}{k+1} \btau}^2_{0,\partial \cT_h} + \norm{\tfrac{h_\cF^{\sfrac{1}{2}}}{k+1} z}^2_{0,\partial \cT_h}+ \norm{\tfrac{h_\cF^{\sfrac{1}{2}}}{k+1} \varphi}^2_{0,\partial \cT_h})^{\sfrac{1}{2}} \abs{ (\underline{\bv}, \underline{\hat{\bv}} )}_{\mathcal{U}\times \hat{\mathcal{U}}},
	\end{equation}
	for all $\boldsymbol{\Phi} \in \cH_2$ such that $\btau \in H^s(\cT_h, \mathbb{R}^{d\times d}_{\text{sym}})$ and $z, \varphi \in H^s(\cT_h)$ with $s\geq \sfrac{1}{2}$, and for all $(\underline{\bv}, \underline{\hat{\bv}} ) \in \mathcal{U}\times \hat{\mathcal{U}}$. The result follows from applying the discrete trace inequality \eqref{discTrace}.
\end{proof}

\begin{proposition}
	Problem~\eqref{sd}-\eqref{initial-R1-R2-h*c} admits a unique solution.
\end{proposition}

\begin{proof}
We introduce the linear operator $\mmean{\cdot}: \mathcal{H}_{1,h} \to \hat{\mathcal{H}}_{1,h}$ defined by $\mmean{ \ubv } = [ \mmean{\bv} \mid \mmean{\bw} \mid \mmean{\bs}]$, where for an interior face $F=K\cap K'$, the mean value is $\mmean{\bv}|_F = \tfrac12 ( \bv|_K + \bv|_{K'} )|_F$. For a boundary face $F \in \cF_h^\partial$, $\mmean{\bv}|_F = \mathbf{0}$ due to the Dirichlet boundary condition for $\bu$, while $\mmean{\bw}|_F = \bw|_F$ and $\mmean{\bs}|_F = \bs|_F$ because $p=0$ and $\theta=0$ on $\Gamma$, respectively, which implies natural boundary conditions for the fluxes $\bq$ and $\br$.

The algebraic differential equation \eqref{sd} can be decomposed into the following system of equations:
\begin{align}\label{ADE1}
\begin{split}
\inner*{\dot{\underline{\bu}}_{h}, \underline{\bv}}_{\cH_1} &+
\inner{\dot{\boldsymbol{\Theta}}_h, \boldsymbol{\Phi} }_{\cH_2} + \inner{\mathfrak I \underline{\bu}_{h},\underline{\bv} }_\Omega
+ B_h(\boldsymbol{\Theta}_h, (\underline{\bv}, \mathbf 0 )) - B_h(\boldsymbol{\Phi}, (\underline{\bu}_{h}, \underline{\hat{\bu}}_{h} ))
\\
& \quad +\dual{ \tfrac{(k+1)^2}{h\cF}(\underline{\bu}_{h} - \underline{\hat{\bu}}_{h} ), \underline{\bv} }_{\partial \cT_h} = \inner*{\underline{\bF}, \underline{\bv}}_\Omega + \inner{g, \varphi}_\Omega, \quad \forall (\underline{\bv} , \boldsymbol{\Phi}) \in \mathcal{H}_{1,h} \times \mathcal{H}_{2,h}
\\
&B_h(\boldsymbol{\Theta}_h, (\mathbf 0, \hat{\underline{\bv} }) ) + \dual{\tfrac{(k+1)^2}{h\cF}(\underline{\bu}_{h} - \underline{\hat{\bu}}_{h} ), - \hat{\underline{\bv} } }_{\partial \cT_h} = 0 \quad \forall \hat{\underline{\bv} } \in \hat{\mathcal{H}}_{1,h}.
\end{split}
\end{align}
From the second equation in \eqref{ADE1}, we deduce that the numerical flux $\hat{\underline{\bu}}_h \in \hat{\mathcal{H}}_{1,h}$ satisfies:
$$
\dual{ \hat{\underline{\bu}}_h, \hat{\underline{\bv} } }_{\cF_h} = \dual{ \mmean{\underline{\bu}_h} - \tfrac{h\cF}{(k+1)^2} \stack*{ \mmean{(\bsig_h - (\alpha p_h + \beta\theta_h) I_d)\bn} \mid \mmean{p_h\bn} \mid \mmean{\theta_h\bn} }, \hat{\underline{\bv} } }_{\cF_h}\quad \forall \hat{\underline{\bv} } \in \hat{\mathcal{H}}_{1,h}.
$$
In other words, $\hat{\underline{\bu}}_h$ is the $L^2$-orthogonal projection of the expression
$$
\mmean{\underline{\bu}_h} - \tfrac{h\cF}{(k+1)^2} \stack*{ \mmean{(\bsig_h - (\alpha p_h + \beta\theta_h) I_d)\bn} \mid \mmean{p_h\bn} \mid \mmean{\theta_h\bn} }
$$
onto $\hat{\mathcal{H}}_{1,h}$. Substituting this expression for $\hat{\underline{\bu}}_h$ into the first equation of \eqref{ADE1} yields a system of ordinary differential equations (ODEs) with unknowns $\ubu_h \in \mathcal{C}^1{[0,T]}(\mathcal{H}_{1,h})$ and $\boldsymbol{\Theta}_h \in \mathcal{C}^1{[0,T]}(\mathcal{H}_{2,h})$, along with the initial conditions:
\begin{equation}\label{initODE}
\ubu_h(0) = \Pi^{k+1}_\cT \ubu^0, \quad \bsig_{h}(0)= \Pi_\cT^k \bsig^0, \quad p_h(0) = \Pi_\cT^k p^0, \quad \text{and} \quad \theta_h(0) = \Pi_\cT^k \theta^0.
\end{equation}
The well-posedness of this resulting first-order ODE system follows directly from the fact that $\inner{\cdot, \cdot}_{\mathcal{H}_1}$ and $\inner{\cdot, \cdot}_{\mathcal{H}_2}$ are inner products on the finite-dimensional function spaces $\mathcal{H}_{1,h}$ and $\mathcal{H}_{2,h}$, respectively.
\end{proof}

We now proceed to verify the consistency of the HDG scheme \eqref{sd} with the continuous problem \eqref{eq:weakform}.

\begin{proposition}\label{consistency}
Let $\underline{\bu} = [\bu \mid \bq \mid \br] \in \mathcal{C}_{[0,T]}^1(\cH_1) \cap \mathcal{C}^0_{[0,T]}(\cX_1)$ and $\boldsymbol{\Theta} = (\bsig , \stack{p \sep \theta}) \in \mathcal{C}^1_{[0,T]}(\cH_2) \cap \mathcal{C}^0_{[0,T]}(\cX_2)$ be the solutions of \eqref{eq:weakform}. Assume that $\bsig - (\alpha p + \beta\theta) I_d \in\cC_{[0,T]}^0(H^s(\cT_h, \bbS))$ and $\underline{\bu} \in \cC_{[0,T]}^0(H^s(\cT_h, \mathbb{R}^{d \times 3}))$, with $s>\sfrac{1}{2}$. Then, the following consistency equation holds:
\begin{align}\label{consistent}
	\begin{split}
		&\inner*{\dot{\underline{\bu}}, \underline{\bv}_h}_{\cH_1} +
		\inner{\dot{\boldsymbol{\Theta}}, \boldsymbol{\Phi}_h }_{\cH_2} + \inner{\mathfrak I \underline{\bu} ,\underline{\bv}_h }_\Omega
		+ B_h(\boldsymbol{\Theta}, (\underline{\bv}_{h}, \underline{\hat{\bv}}_{h} ))
		\\& \quad
	  - B_h(\boldsymbol{\Phi}_h, (\underline{\bu}, \underline{\bu}|_{\partial \cF_h} ))
	  + \dual{ \tfrac{(k+1)^2}{h_\cF}(\underline{\bu} - \underline{\bu}|_{\partial \cF_h}) , \underline{\bv}_{h} - \underline{\hat{\bv}}_{h} }_{\partial \cT_h} = \inner*{\underline{\bF}, \underline{\bv}_h}_\Omega + \inner{g, \varphi_h}_\Omega,
	\end{split}
\end{align}
for all $(\underline{\bv}_{h}, \underline{\hat{\bv}}_{h} ) \in \mathcal{H}_{1,h} \times \hat{\mathcal{H}}_{1,h}$ and $\boldsymbol{\Phi}_h = (\btau_h, \stack{z_h \sep \varphi_h}) \in \cH_{2,h}$.
\end{proposition}

\begin{proof}
	The continuity of the normal components of $\bsig - (\alpha p + \beta\theta) I_d$, $p I_d$, and $\theta I_d$ across element interfaces in $\cT_h$, coupled with the homogeneous Dirichlet boundary conditions on $\Gamma$, yields:
	\begin{align*}
		\begin{split}
			B_h(\boldsymbol{\Theta}, (\underline{\bv}_{h}, \underline{\hat{\bv}}_{h} )) &=
			\inner*{\bsig - (\alpha p + \beta\theta) I_d, \beps(\bv_{h})}_{\cT_h}
			- \dual*{(\bsig - (\alpha p + \beta\theta) I_d)\bn, \bv_{h} - \hat{\bv}_{h} }_{\partial \mathcal{T}_h} \\
			&\qquad - \inner*{p, \div\bw_h }_{\cT_h}
			+ \dual*{p\bn, \bw_h - \hat{\bw}_h }_{\partial \mathcal{T}_h} \\
			&\qquad - \inner*{\theta, \div\bs_h }_{\cT_h}
			+ \dual*{\theta\bn, \bs_h - \hat{\bs}_h }_{\partial \mathcal{T}_h} \\
			& = \sum_{K \in \cT_h} \Big( \inner*{\bsig - (\alpha p + \beta\theta) I_d, \beps(\bv_{h})}_{K} - \inner*{p, \div\bw_h}_{K} - \inner*{ \theta, \div\bs_h}_{K}  \\
			&\qquad  - \dual*{(\bsig - (\alpha p + \beta\theta) I_d)\bn_K, \bv_{h}}_{F_K} + \dual*{p\bn_K, \bw_h}_{F_K} + \dual*{\theta\bn_K, \bs_h}_{F_K} \Big).
		\end{split}
		\end{align*}
		Applying element-wise integration by parts to the right-hand side of the previous identity, and utilizing the identities from the continuous problem \eqref{eq:TP-a}, \eqref{eq:TP-b}, and \eqref{eq:TP-f} in the form 
		$$\stack*{ - \bdiv \bigl(\bsig - (\alpha p + \beta\theta) I_d\bigr) \mid \nabla p \mid \nabla \theta } = \underline{\bF} - \mathfrak I \underline{\bu} - \dot{\underline{\bu}}R,
		$$ 
		we obtain:
		\begin{align}\label{b2}
			\begin{split}
				B_h(\boldsymbol{\Theta}, (\underline{\bv}_{h}, \underline{\hat{\bv}}_{h} )) &= \inner*{-\bdiv\bigl(\bsig - (\alpha p + \beta\theta) I_d\bigr), \bv_{h}}_{\cT_h} + \inner*{\nabla p, \bw_h}_{\cT_h} + \inner*{\nabla \theta, \bs_h}_{\cT_h} \\
				& = \inner*{\underline{\bF}, \underline{\bv}_h}_{\Omega} - \inner*{\mathfrak I \underline{\bu}, \underline{\bv}_h}_{\Omega} - \inner*{\dot{\underline{\bu}} R, \underline{\bv}_h}_{\Omega}.
			\end{split}
		\end{align}

		On the other hand, we have that:
		\begin{equation}\label{b1}
			B_h(\boldsymbol{\Phi}_h, (\underline{\bu}, \underline{\bu}|_{\cF_h})) = \inner*{\btau_h - (\alpha z_h + \beta \varphi_h) I_d, \beps(\bu) }_{\cT_h} -
			\inner*{z_h, \div\bq}_{\cT_h} - \inner*{\varphi_h, \div\br}_{\cT_h} \quad \forall \boldsymbol{\Phi}_h \in \cH_{2,h}.
		\end{equation}
		Substituting \eqref{b2} and \eqref{b1} into the consistency equation \eqref{consistent} gives:
		\begin{align}\label{consistent0}
			\begin{split}
				\inner*{\dot{\underline{\bu}}, \underline{\bv}_h}_{\cH_1} &+
				\inner{\dot{\boldsymbol{\Theta}}, \boldsymbol{\Phi}_h }_{\cH_2} + \inner{\mathfrak I \underline{\bu} ,\underline{\bv}_h }_\Omega
				+ \left( \inner*{\underline{\bF}, \underline{\bv}_h}_{\Omega} - \inner*{\mathfrak I \underline{\bu}, \underline{\bv}_h}_{\Omega} - \inner*{\dot{\underline{\bu}} R, \underline{\bv}_h}_{\Omega} \right)
				\\& \quad
				- \left( \inner*{\btau_h - (\alpha z_h + \beta \varphi_h) I_d, \beps(\bu) }_{\cT_h} -
				\inner*{z_h, \div\bq}_{\cT_h} - \inner*{\varphi_h, \div\br}_{\cT_h} \right)
				\\&
				= \inner*{\underline{\bF}, \underline{\bv}_h}_\Omega + \inner{g, \varphi_h}_\Omega.
			\end{split}
		\end{align}
		Simplifying terms and recalling that $\inner*{\dot{\underline{\bu}} R, \underline{\bv}_h}_{\Omega} = \inner*{\dot{\underline{\bu}}, \underline{\bv}_h}_{\cH_1}$, the equation becomes:
		\begin{align}\label{last}
			\inner{\dot{\boldsymbol{\Theta}}, \boldsymbol{\Phi}_h }_{\cH_2}
			&- \inner*{\btau_h - (\alpha z_h + \beta \varphi_h) I_d, \beps(\bu) }_{\cT_h} + \inner*{z_h, \div\bq}_{\cT_h} + \inner*{\varphi_h, \div\br}_{\cT_h} = \inner{g, \varphi_h}_\Omega.
		\end{align}
		Now, we point out that applying Green's formula to the second equation of \eqref{eq:weakform} yields:
		$$
		\inner{\dot{\boldsymbol{\Theta}}, \boldsymbol{\Phi}}_{\cH_2} - \inner{\btau - (\alpha z + \beta \varphi) I_d, \beps(\bu)}_\Omega + \inner{ z, \div \bq}_\Omega + \inner{ \varphi ,\div \br}_\Omega = \inner{g, \varphi}_\Omega,
		$$
		for all $\boldsymbol{\Phi}\in \cX_2$. Hence, the density of the embedding $\cX_2 \hookrightarrow \cH_{2}$ ensures that equation \eqref{last} holds true, thereby verifying consistency.
\end{proof}

\section{Convergence analysis of the HDG method}\label{sec:convergence}

The convergence analysis of the HDG method \eqref{sd} follows standard procedures. Using the stability of the HDG method and the consistency result \eqref{consistent}, we prove that the projected errors
\begin{align*}
    \begin{alignedat}{2}
        \be_{u, h}(t)      &:= \Pi^{k+1}_\cT \bu  - \bu_{h}, 
		&\quad
		\be_{\hat u, h}(t) &:= \Pi_\cF^{k+1}(\bu|_{\cF_h}) - \hat{\bu}_{h}, \\
		\be_{q,h}(t)       &:= \Pi^{k+1}_\cT \bq  - \bq_{h}, 
		&\quad 
		\be_{\hat q, h}(t) &:= \Pi_\cF^{k+1}(\bq|_{\cF_h}) - \hat{\bq}_{h},
		\\
		\be_{r,h}(t)       &:= \Pi^{k+1}_\cT \br  - \br_{h}, 
		&\quad 
        \be_{\hat r, h}(t) &:= \Pi_\cF^{k+1}(\br|_{\cF_h}) - \hat{\br}_{h}, \\
        \be_{\sigma,h}(t)    &:= \Pi_\cT^k\bsig - \bsig_h, &\quad
        e_{p,h}(t)         &:= \Pi_\cT^kp - p_h, \quad
		e_{\theta,h}(t)      := \Pi_\cT^k\theta - \theta_h.
    \end{alignedat}
\end{align*}
can be estimated in terms of the approximation errors
\begin{align*}
    \begin{alignedat}{2}
        \bchi_{u}(t)       &:= \bu  - \Pi^{k+1}_\cT \bu, &\quad
		\bchi_{\hat u}(t)  &:= \bu|_{\cF_h} - \Pi_\cF^{k+1}(\bu|_{\cF_h}), \\
        \bchi_{q}(t)       &:= \bq  - \Pi^{k+1}_\cT \bq, 
		&\quad
        \bchi_{\hat q}(t)  &:= \bq|_{\cF_h} - \Pi_\cF^{k+1}(\bq|_{\cF_h}),
		\\
		\bchi_{r}(t)       &:= \br  - \Pi^{k+1}_\cT \br, &\quad
		\bchi_{\hat r}(t)  &:= \br|_{\cF_h} - \Pi_\cF^{k+1}(\br|_{\cF_h}),
        \\
        \bchi_\sigma(t)      &:= \bsig - \Pi_\cT^k\bsig, &\quad
        \chi_p(t)           &:= p - \Pi_\cT^kp, \quad
		\chi_\theta(t)      := \theta - \Pi_\cT^k\theta.
    \end{alignedat}
\end{align*}
As before, we concatenate the error terms corresponding to velocities and heat flux by defining
\begin{align*}
    \begin{alignedat}{2}
        \underline{\be}_{u,h}(t)       &:= \stack*{\be_{ u, h} \mid \be_{q, h} \mid \be_{r, h}}, &\quad
        \underline{\be}_{\hat u,h}(t)       &:= \stack*{\be_{\hat u, h} \mid \be_{\hat q, h} \mid \be_{\hat r, h}}, \\
        \underline{\bchi}_u(t)  &:= \stack*{ \bchi_{u} \mid \bchi_{q} \mid \bchi_{r}}, &\quad
        \underline{\bchi}_{\hat u}(t)  &:= \stack*{\bchi_{\hat u}  \mid \bchi_{\hat q} \mid \bchi_{\hat r}}.
    \end{alignedat}
\end{align*}
It will also be useful to introduce the notations $\boldsymbol{\mathcal{E}}_{\Theta,h} \coloneq (\be_{\sigma,h}, \stack{e_{p,h} \sep e_{\theta,h}})$ and $\boldsymbol{\mathcal{X}}_{\Theta} \coloneq (\bchi_\sigma, \stack{\chi_p \sep \chi_\theta})$.

\begin{lemma}\label{stab_sd}
Under the conditions of Proposition 4.3, there exists a constant $C>0$ independent of $h$ and $k$ such that
	\begin{align}\label{stab}
		\begin{split}
			&\max_{[0, T]}\norm{ \underline{\be}_{u,h}}^2_{\cH_1} + \max_{[0, T]} \norm{\boldsymbol{\mathcal{E}}_{\Theta,h}}^2_{\cH_2} + \int_0^T \inner{\mathfrak I \underline{\be}_{u,h}(s), \underline{\be}_{u,h}(s)}_\Omega \,\text{d}s + \int_0^T \norm{\tfrac{k+1}{h_\cF^{\sfrac{1}{2}}}(\underline{\be}_{u,h} - \underline{\be}_{\hat{u}, h} )}^2_{0,\partial \cT_h}\,\text{d}s
			\\&
 		 \leq C \int_0^T \Big( \abs{ (\underline{\bchi}_u,\underline{\bchi}_{\hat u})}^2_{\mathcal{U} \times \hat{\mathcal{U}}} + \norm{\tfrac{h_\cF^{\sfrac{1}{2}}}{k+1} \bchi_\sigma}^2_{0,\partial \cT_h} + \norm{\tfrac{h_\cF^{\sfrac{1}{2}}}{k+1} \chi_p}^2_{0,\partial \cT_h} + \norm{\tfrac{h_\cF^{\sfrac{1}{2}}}{k+1} \chi_\theta}^2_{0,\partial \cT_h}  \Big) \, \text{d}t.
		\end{split}
	\end{align} 
\end{lemma}
\begin{proof}
	 By virtue of the consistency result \eqref{consistent}, it is straightforward that the projected errors satisfy the following orthogonality property:
	 \begin{align}\label{orthog}
		\begin{split}
			& \inner*{\dot{\underline{\be} }_{u,h}, \underline{\bv}}_{\cH_1} +
			\inner*{\dot{\boldsymbol{\mathcal{E}}}_{\Theta,h}, \boldsymbol{\Phi} }_{\cH_2}  + \inner{\mathfrak I \underline{\be}_{u,h},\underline{\bv}}_\Omega
			+ B_h(\boldsymbol{\mathcal{E}}_{\Theta,h}, (\underline{\bv}, \underline{\hat{\bv}} ))
			\\
			&\qquad \qquad  - B_h(\boldsymbol{\Phi}, (\underline{\be}_{u,h}, \underline{\be}_{\hat{u},h} ))
			 +\dual{ \tfrac{(k+1)^2}{h_\cF}(\underline{\be}_{u,h} - \underline{\be}_{\hat{u}, h} ), \underline{\bv} - \underline{\hat{\bv}} }_{\partial \cT_h}
			 \\ &
			  =
			  -\inner*{\dot{\underline{\bchi} }_{u}, \underline{\bv}}_{\cH_1} -
			\inner*{\dot{\boldsymbol{\mathcal{X}}}_{\Theta}, \boldsymbol{\Phi} }_{\cH_2}  - \inner{\mathfrak I \underline{\bchi}_u,\underline{\bv}}_\Omega
			- B_h(\boldsymbol{\mathcal{X}}_{\Theta}, (\underline{\bv}, \underline{\hat{\bv}} ))
			\\ &
			\qquad \qquad   + B_h(\boldsymbol{\Phi}, (\underline{\bchi}_{u}, \underline{\bchi}_{\hat{u}} ))
			 - \dual{ \tfrac{(k+1)^2}{h_\cF}(\underline{\bchi}_{u} - \underline{\bchi}_{\hat{u}} ), \underline{\bv} - \underline{\hat{\bv}} }_{\partial \cT_h}
		\end{split}
	\end{align}
	for all $\boldsymbol{\Phi} = (\btau, \stack{z \sep \varphi}) \in \mathcal H_{2,h}$ and $(\underline{\bv} , \hat{\underline{\bv}} )\in \mathcal{H}_{1,h} \times \hat{\mathcal{H}}_{1,h}$.

	The second line of \eqref{orthog} simplifies using the $L^2$-orthogonality of the projectors:
 	\begin{align*}
		\inner{\dot{\underline{\bchi} }_{u}, \underline{\bv}}_{\cH_1} + \inner{\mathfrak I\underline{\bchi}_u,\underline{\bv}}_\Omega =0 \quad \forall \underline{\bv} \in \cH_{1,h},
	\quad \text{and} \quad
		\inner*{\dot{\boldsymbol{\mathcal{X}}}_{\Theta}, \boldsymbol{\Phi} }_{\cH_2}   =0 \quad \forall \boldsymbol{\Phi}\in \cH_{2,h}.
	\end{align*}
 	Furthermore, for all $(\underline{\bv} , \hat{\underline{\bv}} )\in \cH_{1,h} \times \hat{\cH}_{1,h}$, the following identity holds:
 	\[
		B_h(\boldsymbol{\mathcal{X}}_{\Theta}, (\underline{\bv}, \underline{\hat{\bv}} )) = -\dual*{(\bchi_\sigma - (\alpha \chi_p + \beta \chi_\theta) I_d)\bn, \bv - \hat{\bv}}_{\partial \cT_h} + \dual*{\chi_p\bn, \bw - \hat{\bw}}_{\partial \cT_h} + \dual*{\chi_\theta\bn, \bs - \hat{\bs}}_{\partial \cT_h}.
 	\]
 	This is a direct consequence of the definition of $B_h$ and the fact that $\beps(\cP_{k+1}(\cT_h, \bbR^d)) \subset \cP_k(\cT_h, \bbS)$ and $\div(\cP_{k+1}(\cT_h, \bbR^d)) \subset \cP_k(\cT_h)$.

The choices $\boldsymbol{\Phi} = \boldsymbol{\mathcal{E}}_{\Theta,h}$ and $(\underline{\bv} ,\hat{\underline{\bv} }) = (\underline{\be}_{u,h}, \underline{\be}_{\hat{u},h} )$ in \eqref{orthog} and the Cauchy-Schwarz inequality together with \eqref{Bhh} yield
\begin{align*}
	\frac12 \frac{\text{d}}{\text{d}t} &\Big\{\norm{ \underline{\be}_{u,h}}^2_{\cH_1}
	+ \norm{\boldsymbol{\mathcal{E}}_{\Theta,h}}^2_{\cH_2}  \Big\}
	+ \inner{\mathfrak I \underline{\be}_{u,h},\underline{\be}_{u,h}}_\Omega
	+ \norm{\tfrac{k+1}{h_\cF^{\sfrac{1}{2}}}(\underline{\be}_{u,h} - \underline{\be}_{\hat{u}, h} )}^2_{0,\partial \cT_h} \\
	&= \dual*{(\bchi_\sigma - (\alpha \chi_p + \beta \chi_\theta) I_d)\bn, \be_{u,h} - \be_{\hat{u}_h}}_{\partial \cT_h}
	- \dual*{\chi_p\bn, \be_{q,h} - \be_{\hat{q},h}}_{\partial \cT_h}
	- \dual*{\chi_\theta\bn, \be_{r,h} - \be_{\hat{r},h}}_{\partial \cT_h} \\
	&\quad + B_h(\boldsymbol{\mathcal{E}}_{\Theta,h}, (\underline{\bchi}_{u}, \underline{\bchi}_{\hat{u}} ))
	- \dual{ \tfrac{(k+1)^2}{h_\cF}(\underline{\bchi}_{u} - \underline{\bchi}_{\hat{u}} ), \underline{\be}_{u,h} - \underline{\be}_{\hat{u},h} }_{\partial \cT_h} \\
	&\leq \norm{\tfrac{h_\cF^{\sfrac{1}{2}}}{k+1} (\bchi_\sigma - (\alpha \chi_p + \beta \chi_\theta) I_d)\bn}_{0,\partial \cT_h}
	\norm{\tfrac{k+1}{h_\cF^{\sfrac{1}{2}}} (\be_{u,h} - \be_{\hat{u}_h})}_{0,\partial \cT_h} \\
	&\quad + \norm{\tfrac{h_\cF^{\sfrac{1}{2}}}{k+1} \chi_p \bn}_{0,\partial \cT_h}
	\norm{\tfrac{k+1}{h_\cF^{\sfrac{1}{2}}} (\be_{q,h} - \be_{\hat{q},h})}_{0,\partial \cT_h}  + \norm{\tfrac{h_\cF^{\sfrac{1}{2}}}{k+1} \chi_\theta \bn}_{0,\partial \cT_h}
	\norm{\tfrac{k+1}{h_\cF^{\sfrac{1}{2}}} (\be_{r,h} - \be_{\hat{r},h})}_{0,\partial \cT_h} \\
	&\quad + C \norm{\boldsymbol{\mathcal{E}}_{\Theta,h}}_{\cH_2}
	\abs{ (\underline{\bchi}_u,\underline{\bchi}_{\hat u})}_{\mathcal{U} \times \hat{\mathcal{U}}}  + \norm{\tfrac{k+1}{h_\cF^{\sfrac{1}{2}}} (\underline{\bchi}_{u} - \underline{\bchi}_{\hat{u}})  }_{0,\partial \cT_h}
	\norm{\tfrac{k+1}{h_\cF^{\sfrac{1}{2}}} (\underline{\be}_{u,h} - \underline{\be}_{\hat{u},h})  }_{0,\partial \cT_h}.
\end{align*}
We notice that, because of assumption \eqref{initial-R1-R2-h*c}, the projected errors satisfy vanishing initial conditions, namely, $\be_{\sigma,h}(0) = \mathbf 0$, $e_{p,h}(0) = 0$, $e_{\theta,h}(0) = 0$ and $(\underline{\be}_{u,h}(0), \underline{\be} _{\hat u,h}(0)) = (\mathbf 0, \mathbf 0)$. Hence, integrating over $t\in (0, T]$ and using again the Cauchy-Schwarz inequality we deduce that
\begin{align*}
 		&\norm{ \underline{\be}_{u,h}}^2_{\cH_1} + \norm{\boldsymbol{\mathcal{E}}_{\Theta,h}}^2_{\cH_2} + \int_0^T \inner{\mathfrak I \underline{\be}_{u,h}(s), \underline{\be}_{u,h}(s)}_\Omega \,\text{d}s+ \int_0^T \norm{\tfrac{k+1}{h_\cF^{\sfrac{1}{2}}}(\underline{\be}_{u,h} - \underline{\be}_{\hat{u}, h} )}^2_{0,\partial \cT_h}\,\text{d}s
 		\\
 		&\qquad \lesssim \Big(\int_0^T \Big(\norm{\tfrac{h_\cF^{\sfrac{1}{2}}}{k+1} (\bchi_\sigma - (\alpha \chi_p + \beta \chi_\theta) I_d)\bn}^2_{0,\partial \cT_h} + \norm{\tfrac{h_\cF^{\sfrac{1}{2}}}{k+1} \chi_p \bn}^2_{0,\partial \cT_h} + \norm{\tfrac{h_\cF^{\sfrac{1}{2}}}{k+1} \chi_\theta \bn}^2_{0,\partial \cT_h} \Big)  \text{d}t\Big)^{\sfrac{1}{2}}
		\\ & \qquad \qquad \qquad \qquad \qquad \qquad \qquad \qquad \qquad \qquad \qquad \qquad
		\times
 		\Big( \int_0^T \norm{\tfrac{k+1}{h_\cF^{\sfrac{1}{2}}} (\underline{\be}_{u,h} - \underline{\be}_{\hat{u},h})  }^2_{0,\partial \cT_h}\ \text{d}t\Big)^{\sfrac{1}{2}}
 		\\
 		&\qquad \qquad + \Big(\int_{0}^{T}\norm{\boldsymbol{\mathcal{E}}_{\Theta,h}}^2_{\cH_2} \text{d}t\Big)^{\sfrac12} \Big( \int_0^T \abs{ (\underline{\bchi}_u,\underline{\bchi}_{\hat u})}^2_{\mathcal{U} \times \hat{\mathcal{U}}}  \text{d}t \Big)^{\sfrac12},\quad \forall t\in (0, T].
\end{align*}
Finally, a simple application of Young's inequality yields
\begin{align*}
 		 &\max_{[0, T]}\norm{ \underline{\be}_{u,h}}^2_{\cH_1} + \max_{[0, T]} \norm{\boldsymbol{\mathcal{E}}_{\Theta,h}}^2_{\cH_2} + \int_0^T \inner{\mathfrak I \underline{\be}_{u,h}(s), \underline{\be}_{u,h}(s)}_\Omega \,\text{d}s+ \int_0^T \norm{\tfrac{k+1}{h_\cF^{\sfrac{1}{2}}}(\underline{\be}_{u,h} - \underline{\be}_{\hat{u}, h} )}^2_{0,\partial \cT_h}\,\text{d}s
 		\\
 		& \lesssim \int_0^T \Big(\norm{\tfrac{h_\cF^{\sfrac{1}{2}}}{k+1} (\bchi_\sigma - (\alpha \chi_p + \beta \chi_\theta) I_d)\bn}^2_{0,\partial \cT_h} + \norm{\tfrac{h_\cF^{\sfrac{1}{2}}}{k+1} \chi_p \bn}^2_{0,\partial \cT_h} + \norm{\tfrac{h_\cF^{\sfrac{1}{2}}}{k+1} \chi_\theta \bn}^2_{0,\partial \cT_h} + \abs{ (\underline{\bchi}_u,\underline{\bchi}_{\hat u})}^2_{\mathcal{U} \times \hat{\mathcal{U}}} \Big) \, \text{d}t,
\end{align*}
and the result follows.
\end{proof}

As a consequence of the stability estimate \eqref{stab}, we immediately have the following convergence result for the HDG method \eqref{sd}-\eqref{initial-R1-R2-h*c}.
\begin{theorem}\label{hpConv}
Let
\[
\underline{\bu} = \stack*{\bu \mid \bq \mid \br} \in \mathcal{C}_{[0,T]}^1(\cH_1) \cap \mathcal{C}^0_{[0,T]}(\cX_1)\quad \text{and} \quad \boldsymbol{\Theta} = (\bsig, \stack{p \sep \theta}) \in \mathcal{C}^1_{[0,T]}(\cH_2) \cap \mathcal{C}^0_{[0,T]}(\cX_2)
\]
be the solutions of \eqref{eq:weakform} with appropriate initial conditions. Assume that $\bsig \in \cC^0_{[0,T]}(H^{1+r}(\Omega, \mathbb{R}^{d\times d}_{\text{sym}}))$, $p \in \cC^0_{[0,T]}(H^{1+r}(\Omega))$, $\theta \in \cC^0_{[0,T]}(H^{1+r}(\Omega))$, and $\underline{\bu} \in \cC^0_{[0,T]}(H^{2+r}( \Omega, \bbR^{d \times 3}))$, with $r\geq 0$. Then, there exists a constant $C>0$ independent of $h$ and $k$ such that
\begin{align*}
		 &\max_{[0, T]}\norm{ (\underline{\bu} - \underline{\bu}_h)(t)}_{\cH_1} + \max_{[0, T]} \norm{(\boldsymbol{\Theta} - \boldsymbol{\Theta}_h)(t)}_{\cH_2} + \left(  \int_0^T \norm{\tfrac{k+1}{h_\cF^{\sfrac{1}{2}}}(\underline{\bu}_h - \underline{\hat{\bu}}_{h} )}^2_{0,\partial \cT_h}\,\text{d}s \right)^{\sfrac{1}{2}}
		\\
		&\quad \leq  C \tfrac{h_K^{\min\{ r, k \}+1}}{(k+1)^{r+\sfrac12}} \Big( \max_{[0,T]}\norm{\underline{\bu} }_{2+r,\Omega} + \max_{[0, T]}\norm*{\bsig}_{1+r, \Omega} + \max_{[0, T]}\norm*{p}_{1+r, \Omega} + \max_{[0, T]}\norm*{\theta}_{1+r, \Omega} \Big)\quad \forall k\geq 0.
\end{align*}
\end{theorem}
\begin{proof}
		It follows from the triangle inequality and \eqref{stab} that 
	\begin{align*}
		&\max_{[0, T]}\norm{ (\underline{\bu} - \underline{\bu}_h)(t)}_{\cH_1} 
		+ \max_{[0, T]} \norm{(\bsig - \bsig_h, p - p_h)(t)}_{\cH_2} 
		 + \left( \int_0^T \norm{\tfrac{k+1}{h_\cF^{\sfrac{1}{2}}}
		((\underline{\bu} - \underline{\bu}_h) - ( \underline{\bu} - \underline{\hat{\bu}}_{h}) )}^2_{0,\partial \cT_h}\,\text{d}s \right)^{\sfrac{1}{2}} \\
		&\lesssim \max_{[0, T]}\norm{\underline{\bchi}_u(t)}^2_{\cH_1} 
		+ \max_{[0, T]}\norm{(\bchi_\sigma, \chi_p)(t)}^2_{\cH_2} 
		\\
		&\qquad \qquad + \left( \int_0^T \Big(\norm{\tfrac{h_\cF^{\sfrac{1}{2}}}{k+1} \bchi_\sigma}^2_{0,\partial \cT_h}  
		+ \norm{\tfrac{h_\cF^{\sfrac{1}{2}}}{k+1} \chi_p}^2_{0,\partial \cT_h}  
		+ \abs{ (\underline{\bchi}_u,\underline{\bchi}_{\hat u})}^2_{\mathcal{U} \times \hat{\mathcal{U}}} \Big) \, \text{d}t \right)^{\sfrac{1}{2}},
	\end{align*}
	and the result follows directly from the error estimates \eqref{tool1} and \eqref{tool2}.
\end{proof}

\begin{remark}\label{R1}
	The convergence analysis in Theorem~\ref{hpConv} leads to two main observations:
	\begin{itemize}
	\item The theoretical bound is sub-optimal by a factor \((k+1)^{1/2}\).  This slight deterioration has already been observed in previous works such as Houston et al.~\cite{houston2002}.
	\item The numerical experiments of Subsection~\ref{example1} (Figures~\ref{fig1} and \ref{fig2}) indicate optimal convergence rates for $\bu$, $\bq$, and $\br$  in the $L^2$-norms, suggesting a gap between theoretical predictions and practical performance that will be addressed in future work.
	\end{itemize}
\end{remark}

\section{Numerical results}\label{sec:numresults}

The numerical results presented in this section have been implemented using the finite element library \texttt{Netgen/NGSolve} \cite{schoberl2014c++},  and all simulations were run on a Mac Studio with an Apple M2 Max chip and 32 GB of RAM. 

\subsection{Validation of the convergence rates}\label{example1}
Firstly, we confirm the accuracy of our HDG scheme by analyzing a problem with a manufactured solution. Then, we consider a practical model problem inspired by \cite{antoniettiIMA, morency}.

To confirm the decay of error with respect to the parameters $h$ and $k$,  we employ successive levels of refinement on an unstructured mesh and compare the computed solutions to an exact solution of problem \eqref{eq:TP} given by 
\begin{align}\label{exactSol}
\begin{split}
p(x,y,t) &= \sin(\pi x)\sin(\pi y) \cos(2 \pi t), \quad \theta = \sin(\pi y)\cos(2 \pi t),
\\
\text{and} \quad \bu(x,y,t) &= \begin{pmatrix}
	2 \pi \cos(2 \pi t) \cos(\pi y) \sin(\pi x)
	\\
	2 \pi \cos(2 \pi t) \cos(\pi x) \sin(\pi y)
	\end{pmatrix} \quad \text{in $\Omega\times (0, T]$},
\end{split}
\end{align}
where $\Omega = (0,1)\times (0, 1)$. We assume that the medium characterized by the constitutive law \eqref{eq:TP-e}  is isotropic, namely, the fourth-order elastic tensor $\cC$  is defined by 
\begin{equation}\label{eq:hooke} 
\cC \zeta = 2 \mu\boldsymbol{\zeta} + \lambda \tr(\boldsymbol{\zeta}) I,
\end{equation}
in terms of Lamé coefficients $\mu>0$ and $\lambda>0$. 

We compute the source terms corresponding to the manufactured solution \eqref{exactSol} of \eqref{eq:TP} with the material parameters given by \eqref{L1} or \eqref{L2}. We prescribe non-homogeneous boundary conditions in \eqref{BC} with $\Gamma_D^S
	=\Gamma_D^F =\Gamma_D^H = \Gamma$. 

In our first test, the parameters are chosen as 
\begin{gather}\label{L1}
	\begin{split}
		\rho_S = \rho_F = 1, \quad \phi = 0.5 \quad \nu = 2.0\quad \mu = 50,\quad  \lambda = 100, 
		\\   
		\eta = \kappa =1, \quad \alpha = \beta = 1, \quad a_0 = c_0 = 1, \quad b_0 = 0.5,  
	\\
	 \chi = 1, \quad \tau = 1.
	\end{split}
\end{gather}

We partition the time interval  $[0, T]$ into uniform subintervals of length $\Delta t$. We discretize in time using the Crank-Nicolson method. Owing to its second-order accuracy in time, we choose the time step such that $\Delta t \approx O(h^{(k+2)/2})$. This ensures that the temporal discretization error remains asymptotically smaller than the spatial error, thereby preserving the expected spatial convergence rates. To verify the accuracy of our method, we evaluate the following $L^2$-norm error measures at the final time step $L$:  
\begin{align}\label{Errors1}
	\begin{split}
		\mathtt{e}^{L}_{hk}(\bsig, p, \theta) := \norm*{ \big( \bsig(T) - \bsig_h^L, \stack{p - p_h^L \sep \theta - \theta_h^L} \big) }_{\cH_2}
 \quad 
 \mathtt{e}^L_{hk}(\underline{\bu} )  := \norm{\underline{\bu}(T)  - \underline{\bu}^L_h}_{\cH_1}.
	\end{split}
\end{align}
 
In Figure~\ref{fig1}, we present the errors as functions of the mesh size $h$ for three different polynomial degrees $k$. The $L^2$-errors \eqref{Errors1} are displayed in log-log plots, with the expected rates of convergence represented by dashed lines. The results show that the errors $\mathtt{e}^{L}_{hk}(\bsig, p, \theta)$ and $\mathtt{e}^L_{hk}(\underline{\bu} )$ achieve the optimal convergence rates of $O(h^{k+1})$ and $O(h^{k+2})$, respectively. 

\begin{figure}[!ht]
	\centering
	\includegraphics[width=\textwidth, height=0.2\textheight]{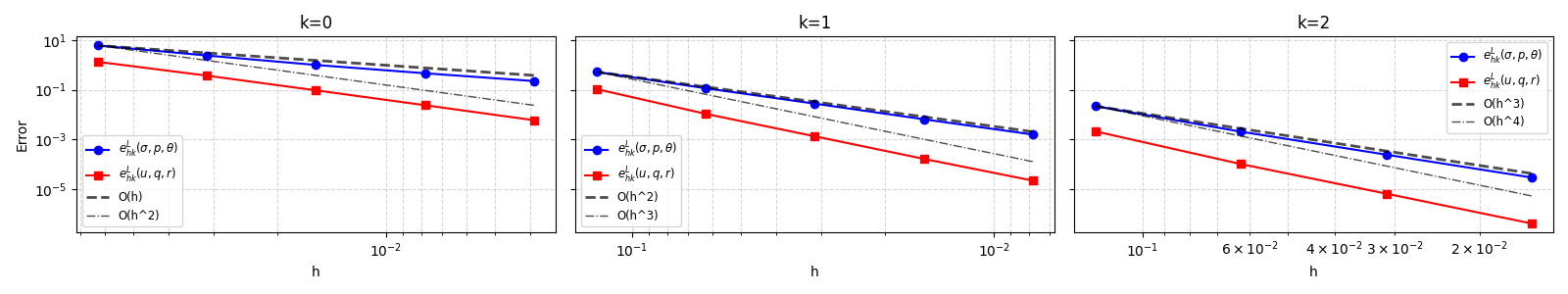}
	\caption{The errors \eqref{Errors1} are plotted against the mesh size h for various polynomial degrees k, using temporal over-refinements. Problem \eqref{eq:TP} is defined with coefficients \eqref{L1} and the exact solution \eqref{exactSol}. }
	\label{fig1}
\end{figure}

To verify the accuracy and stability of the scheme for nearly incompressible poroelastic media with small storage coefficient $s$, we repeat the same experiment using the following set of material parameters:
\begin{gather}\label{L2}
	\begin{split}
		\rho = 1, \quad  \nu_{\cC} = 0.49,\quad   E_{\cC} = 100,\quad \nu_{\cD} = 0.4999,\quad   E_{\cD} = 1000, 
	\\
	\alpha = 1, \quad \beta = 1, \quad \omega=1,\quad \chi = 1, \quad s = 10^{-6}.
	\end{split}
\end{gather}
We recall that the Lamé coefficients are related to the Young's modulus $E$ and Poisson's ratio $\nu$ by the relations $\lambda = \frac{E \nu}{(1+\nu)(1-2\nu)}$ and $\mu = \frac{E}{2(1+\nu)}$.

The error decay for this case is shown in Figure~\ref{fig2}. These results demonstrate the robustness of the proposed HDG scheme in producing accurate approximations for the chosen extreme parameter values. The results are consistent with those obtained in Figure~\ref{fig1} for the coefficients in \eqref{L1}.

\begin{figure}[!ht]
	\centering
	\includegraphics[width=\textwidth, height=0.2\textheight]{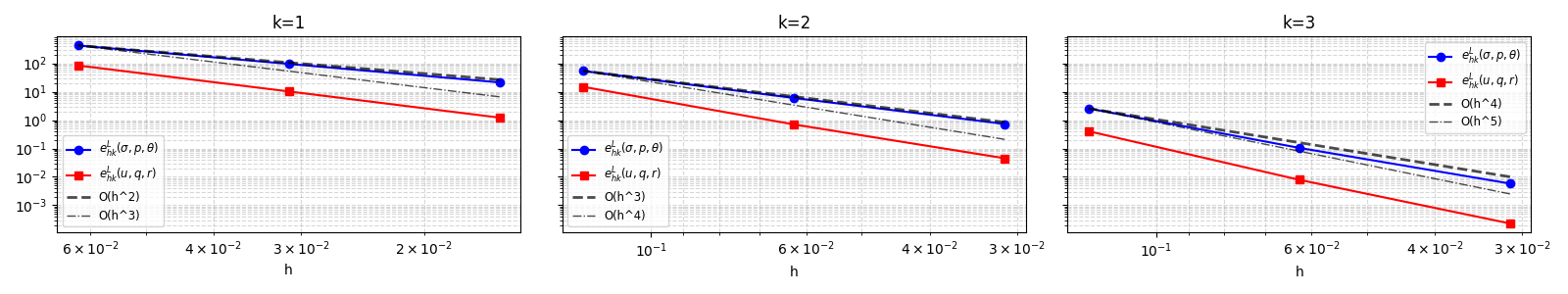}
	\caption{The errors \eqref{Errors1} are plotted against the mesh size h for various polynomial degrees k, using temporal over-refinements. Problem \eqref{eq:TP} is defined with coefficients \eqref{L2} and the exact solution \eqref{exactSol}.}
	\label{fig2}
\end{figure}

Having established the scheme's behavior with respect to the mesh size $h$, we now turn to testing its performance with respect to the parameters $k$ and $\Delta t$.

We fix the space mesh size $h = 1/4$ and the time step $\Delta t = 10^{-6}$ and let the polynomial degree $k$ vary from 1 to 7. In Figure~\ref{T3} we present the errors $\mathtt{e}^{L}_{hk}(\bsig,p)$ and $\mathtt{e}^L_{hk}(\bu, \bp)$ at  $t = 0.3$  plotted against the polynomial degree $k$ on a semi-logarithmic scale. This example uses the same manufactured solution derived from \eqref{exactSol} with material coefficients given in \eqref{L1}. As expected, exponential convergence is observed.

In Figure~\ref{T4}, we show the convergence results obtained by fixing the spatial mesh size at $h=1/16$ and the polynomial degree at $k=3$, while varying the time step $\Delta t$ used to uniformly subdivide the time interval $[0,T]$ with $T=0.5$. The expected convergence rate of $O(\Delta t^2)$ is achieved as the time step is refined.

\begin{table}[!ht]
	\begin{minipage}{0.47\linewidth}
		\centering
  \includegraphics[width=0.9\textwidth]{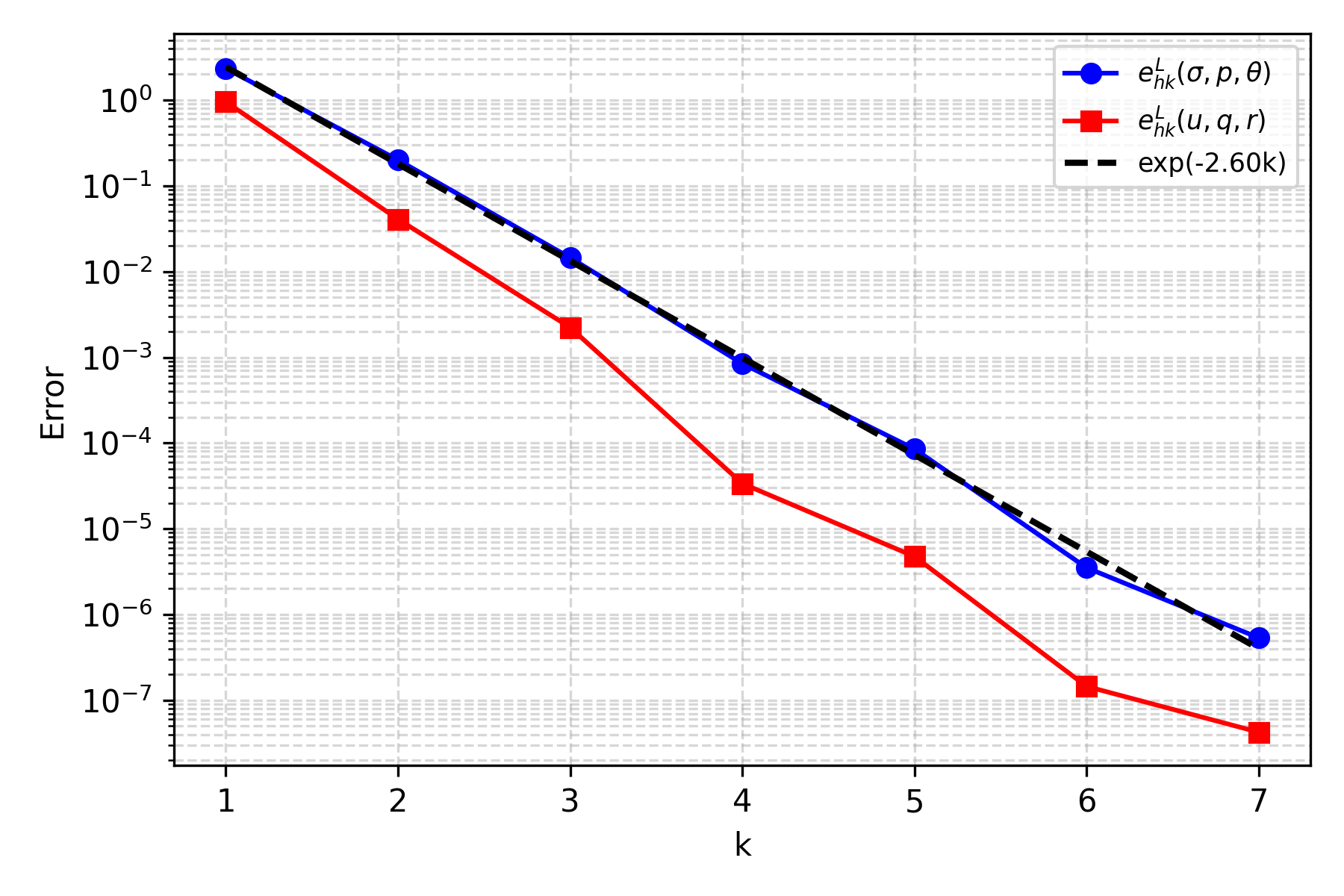}
		\captionof{figure}{Computed errors versus the polynomial degree $k$ with $h=1/4$ and $\Delta t = 10^{-6}$. The errors are measured at $t=0.3$, by employing the coefficients \eqref{L1}.}
		\label{T3}
	  \end{minipage}
	\hfill
	\begin{minipage}{0.47\linewidth}
	  \centering
\includegraphics[width=0.9\textwidth]{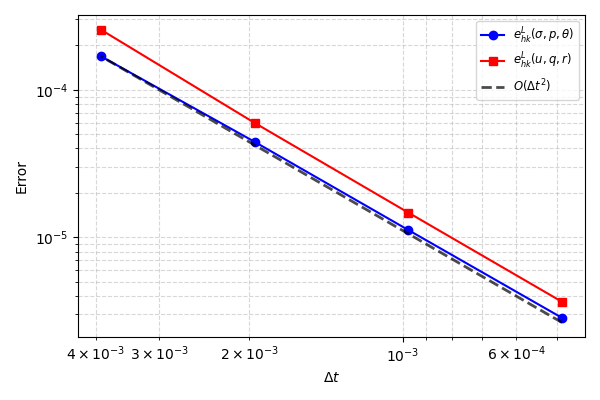}
	  \captionof{figure}{Computed errors for a sequence of uniform refinements in time with $h=1/16$ and $k=3$. The errors are measured at $t=0.5$, with the coefficients \eqref{L1}. }
	  \label{T4}
	\end{minipage}
  \end{table}

\subsection{Wave propagation in a thermoelastic medium}
 
\begin{figure}[!ht]
	\begin{center}
	\includegraphics[scale=0.37]{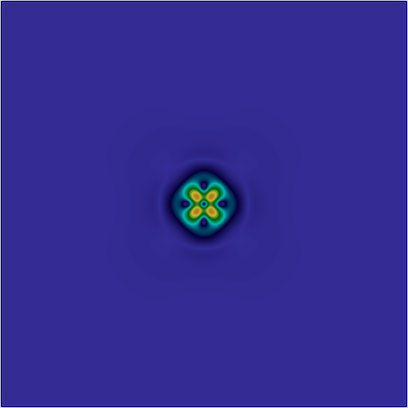}
	\includegraphics[scale=0.37]{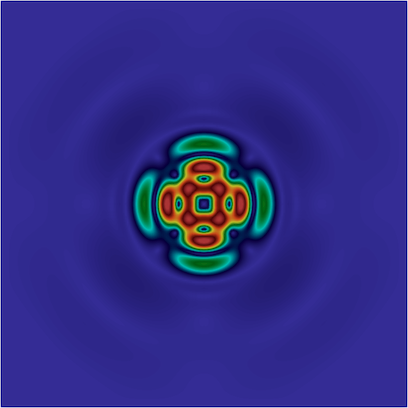}
	\includegraphics[scale=0.37]{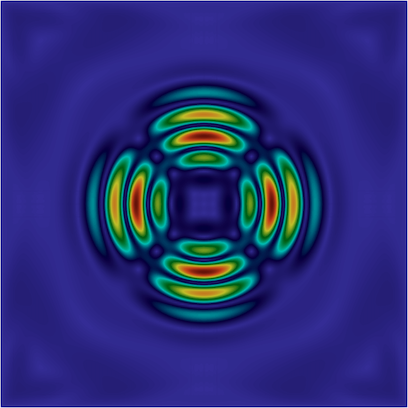}
	\\
	\includegraphics[width=0.5\textwidth]{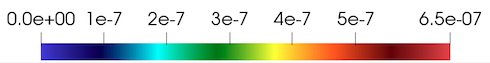}
	\\
	\includegraphics[scale=0.37]{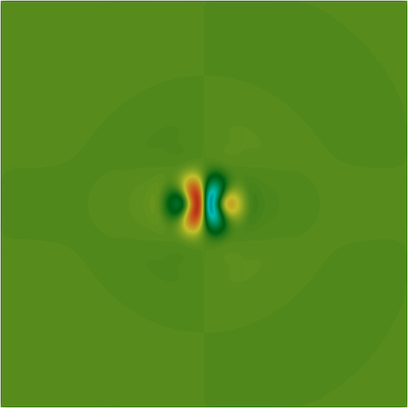}
	\includegraphics[scale=0.37]{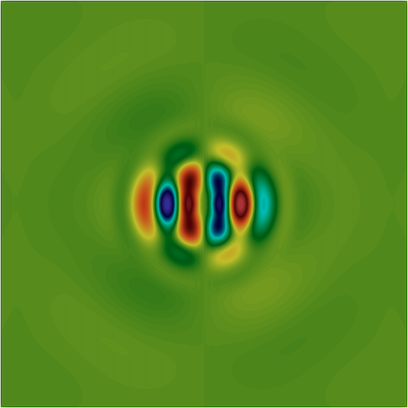}
	\includegraphics[scale=0.37]{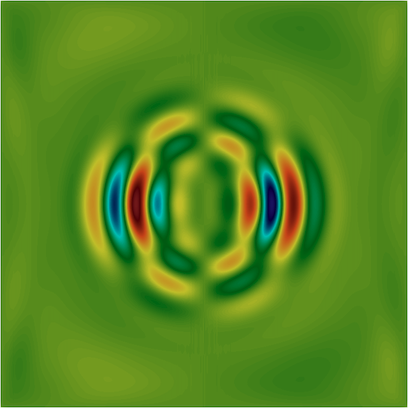}
	\\
	\includegraphics[width=0.5\textwidth]{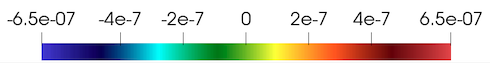}
	\end{center}
	\vspace{-4mm}
	\caption{
	Snapshots of the  solid velocity magnitude $|\bu|$  (top row) and its $y$-component $u_2$ (bottom row) at times 0.2\,\unit{s}, 0.4\,\unit{s}, and 0.6\,\unit{s} (left to right panels). Problem \eqref{eq:TP} is solved using the source term \eqref{Fshear}, with vanishing initial and Dirichlet boundary conditions, and parameter set \eqref{L3}, with $h=50$, $k=7$, and $\Delta t = 10^{-2}$.}
	\label{fig:veloS}
	\end{figure}

	\begin{figure}[!ht]
	\begin{center}
	\includegraphics[scale=0.37]{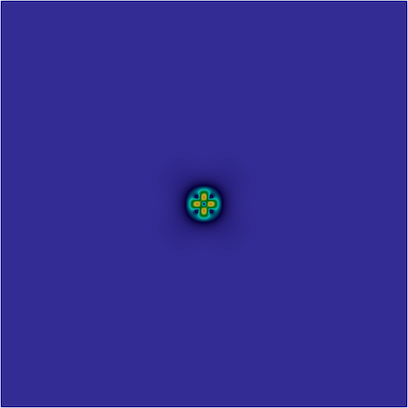}
	\includegraphics[scale=0.37]{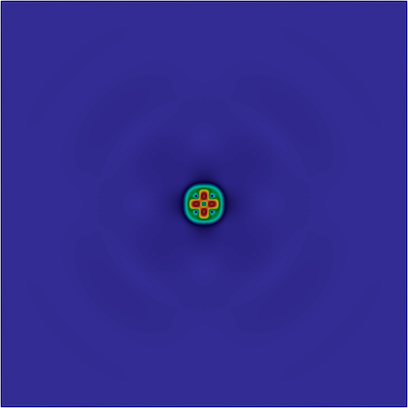}
	\includegraphics[scale=0.37]{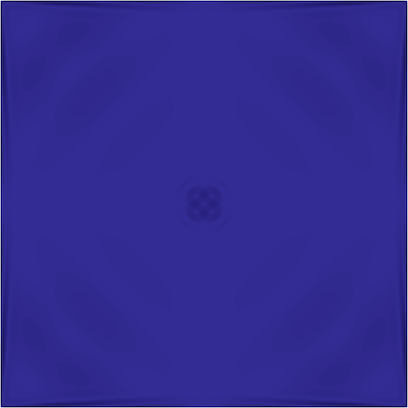}
	\\
	\includegraphics[width=0.5\textwidth]{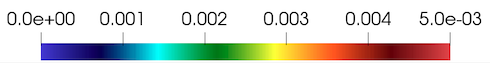}
	\\
	\includegraphics[scale=0.37]{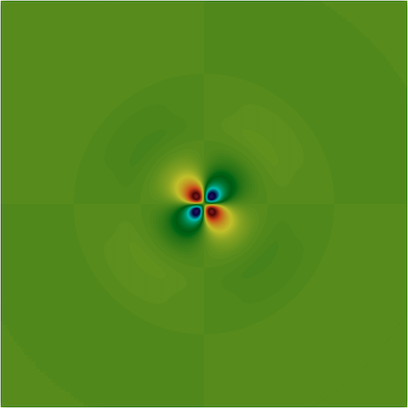}
	\includegraphics[scale=0.37]{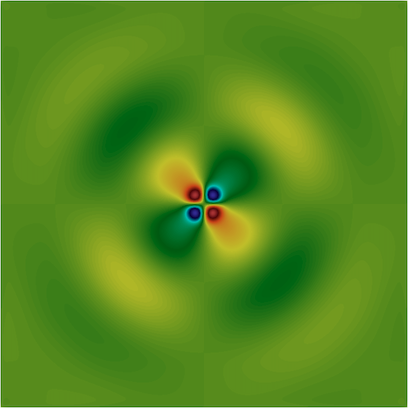}
	\includegraphics[scale=0.37]{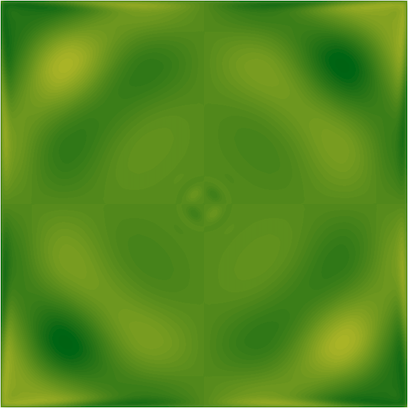}
	\\
	\includegraphics[width=0.5\textwidth]{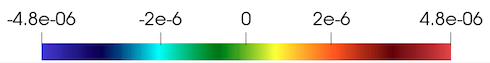}
	\end{center}
	\vspace{-4mm}
	\caption{
	Snapshots of the heat flux magnitude $|\br|$  (top row) and the  temperature $\theta$ (bottom row) at times 0.2\,\unit{s}, 0.4\,\unit{s}, and 0.6\,\unit{s} (left to right panels). Problem \eqref{eq:TP} is solved using the source term \eqref{Fshear}, with vanishing initial and Dirichlet boundary conditions, and parameter set \eqref{L3}, with $h=50$, $k=7$, and $\Delta t = 10^{-2}$.}
	\label{fig:veloT}
	\end{figure}

In this test, we reproduce the numerical simulation from \cite{bonetti2023numerical} of thermal-wave propagation in a homogeneous, isotropic thermoelastic medium. We consider a wave propagation problem in a computational domain $\Omega = (0, 1500)\times (0, 1500)\, \unit{m^2}$ containing an explosive source positioned at $\boldsymbol{x}_s = (750,750) \, \unit{m}$. 

The spatial source function is regularized using a Gaussian approximation to the Dirac delta function:
\[
  \boldsymbol{b}(x,y) = \frac{1}{\epsilon^2}
        \exp\left(-\frac{\norm{\bx - \bx_s}^2}{2\epsilon^2}\right) (\bx - \bx_s),
\]
where $\epsilon = h/3$ provides mesh-dependent smoothing that ensures the regularized delta maintains sufficient support across multiple elements while preserving numerical stability. The forcing term is formulated as:
\begin{equation}\label{Fshear}
    \bF_s =  S(t)\begin{pmatrix}
	0 & 1 \\ 1 & 0
	\end{pmatrix} \boldsymbol{b}(x,y),
\end{equation} 
where $S(t) = A_0 \cos(2\pi f_0(t - t_0)) \exp(-2f_0^2(t - t_0)^2)$ models a compressive seismic pulse centered at $t_0 = 0.3$ s, with amplitude $A_0 = 10$ and peak-frequency $f_0 = 5$ \unit{Hz}. This configuration generates a shear source centered at the center of the square domain.

The physical parameters are set according to \eqref{L3}, representing a typical geophysical medium with fluid-saturated porosity $\phi = 0.3$ and Maxwell-Cattaneo relaxation time $\tau = 1.5 \times 10^{-2}\,\unit{s}$ ensuring finite-speed thermal wave propagation.

\begin{gather}\label{L3}
    \begin{split}
		\rho_F = 1000\,\unit{kg/m^3},\quad \rho_S = 2650\,\unit{kg/m^3},\quad \phi = 0.3 , \quad\nu = 2.0,
    \\
        \mu = 1.885 \cdot 10^9\,\unit{Pa},\quad \lambda = 4.433 \cdot 10^8\,\unit{Pa},
    \\
    \tau = 1.5\cdot 10^{-2}\,\unit{s}, \quad \chi = 1.5\cdot 10^{4}\,\unit{m^2 Pa/K^2 s},\quad \frac{\eta}{\kappa} = 10^{9}\,\unit{m^2 /Pa},
    \\
	c_0 = 4.1695\,\unit{Pa/K^2}, \quad b_0 = 1.3684 \cdot 10^{-5}\,\unit{K^{-1}}, \quad a_0 = 1.3684 \cdot 10^{-10}\,\unit{Pa^{-1}},
    \\
    \alpha = 0.7143\quad \text{and}, \quad \beta = 4.8571 \cdot 10^{4}\,\unit{Pa/K}.
    \end{split}
\end{gather}

Figure~\ref{fig:veloS} illustrates the temporal evolution of the solid velocity magnitude $|\mathbf{u}|$ and its $y$-component $u_2$ at times $t = 0.2$, $0.4$, and $0.6$~\unit{s}. The applied punctual shear force, represented by an antisymmetric moment tensor, generates the characteristic quadrupolar radiation pattern observed in elastic media. The symmetry axes orient along the domain diagonals, directly reflecting the antisymmetric structure of the applied moment tensor source. This cloverleaf-like distribution surrounding the source location represents the expected theoretical behavior for shear-dominated deformation in homogeneous thermo-poroelastic media.

The temporal evolution of the velocity magnitude (Figure~\ref{fig:veloS}, top row) reveals the complex wave dynamics inherent to the coupled thermo-poroelastic system. At the initial time $t = 0.2$~\unit{s}, the quadrupolar pattern remains concentrated near the source, showing the characteristic symmetry of shear wave radiation. As time progresses to $t = 0.4$~\unit{s}, the expanding wavefront begins to exhibit different wave types: the outer, rapidly-advancing front corresponds to the fast compressional ($P$-) wave, while the inner, more slowly-propagating front represents the shear ($S$-) wave. By $t = 0.6$~\unit{s}, the fast $P$-wave reaches the domain boundaries, producing observable corner reflections characteristic of finite domains with Dirichlet boundary conditions. The absence of a clearly visible slow P-wave (Biot wave) in our results is consistent with theoretical expectations for viscous fluid regimes, where this mode typically exhibits extremely low amplitude and strong attenuation.

The $y$-component of velocity (Figure~\ref{fig:veloS}, bottom row) provides enhanced visualization of the shear wave's directional characteristics. The antisymmetric distribution with respect to the $y$-axis clearly demonstrates the dipolar nature of shear wave motion, confirming the theoretical expectation for moment tensor sources. 

Figure~\ref{fig:veloT} illustrates the temporal evolution of the heat flux magnitude $|\br|$ and temperature $\theta$ fields at identical time instants. The heat flux distribution (top row) displays a quadrupolar radiation pattern similar to the mechanical velocity fields, but with distinctly different propagation characteristics and amplitude scaling. The initial quadrupolar structure at $t = 0.2$~\unit{s} gradually evolves and disperses as thermal diffusion becomes dominant.

The temperature field evolution (bottom row of Figure~\ref{fig:veloT}) reveals the presence of the thermal ($T$-) wave, which exhibits  different propagation behavior compared to the elastic wave modes. The thermal wave propagates significantly more slowly than both the compressional and shear waves due to the diffusive nature of heat transport. The temperature distribution remains more spatially concentrated around the source region throughout the simulation, reflecting the lower thermal diffusivity relative to the elastic wave speeds. 

A qualitative comparison with the results in \cite[Figures 6-7]{bonetti2023numerical}  reveals excellent agreement between the HDG solution and the reference PolyDG method. The wavefront shapes and amplitudes, as well as the appearance of shear and compressional modes, are faithfully reproduced.

\subsection{Comparison of poroelastic and thermo-poroelastic wave propagation}

\begin{figure}[t!]
\begin{center}
\includegraphics[width=0.32\textwidth]{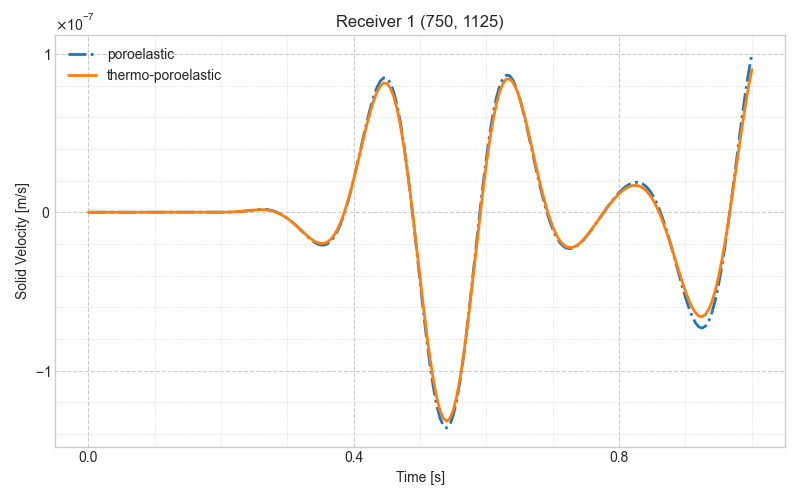}
\includegraphics[width=0.32\textwidth]{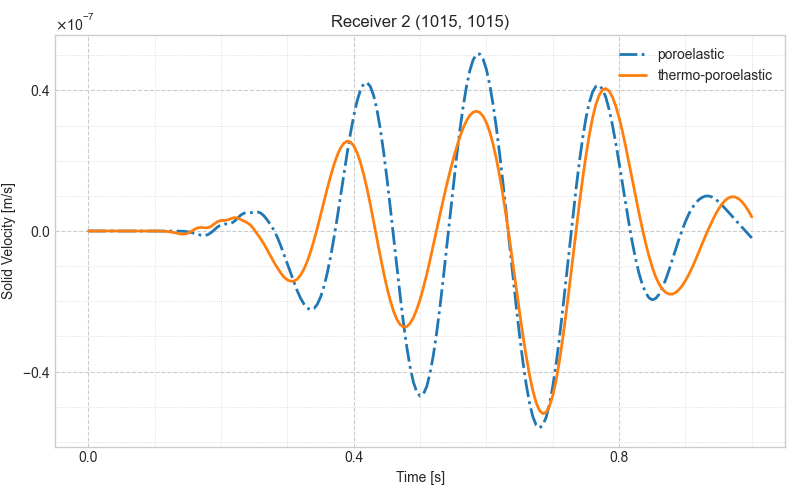}
\includegraphics[width=0.32\textwidth]{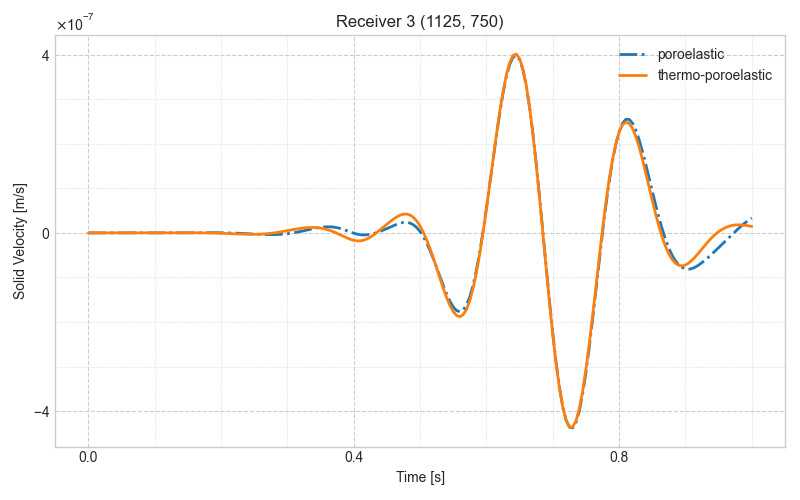}
\end{center}
\vspace{-4mm}
\caption{
The figure illustrates transient behaviour of the $y$-components of solid velocity $\bu$ at the source receivers   $\boldsymbol{x}_{r1} = (750, 1125)$, $\boldsymbol{x}_{r2} = (1015, 1015)$ $\boldsymbol{x}_{r3} = (1125, 750)$ \unit{m} (left to right). We compare the thermo-poroelastic problem \eqref{eq:TP} shown with continuous lines to a purely poroelastic model problem represented by dashed lines.  We use the source term \eqref{Fshear}, and the parameter set \eqref{L3}, with $h=50$, $k=5$, and $\Delta t = 0.005$.}
\label{fig:heteroS}
\end{figure}

\begin{figure}[t!]
\begin{center}
\includegraphics[width=0.32\textwidth]{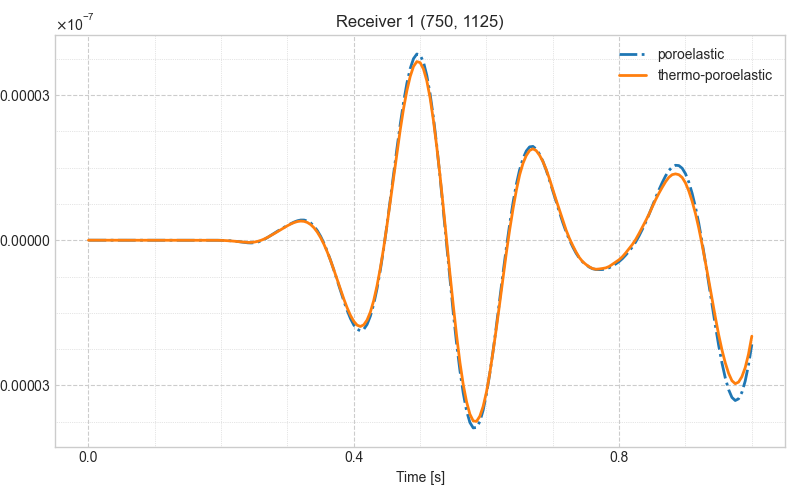}
\includegraphics[width=0.32\textwidth]{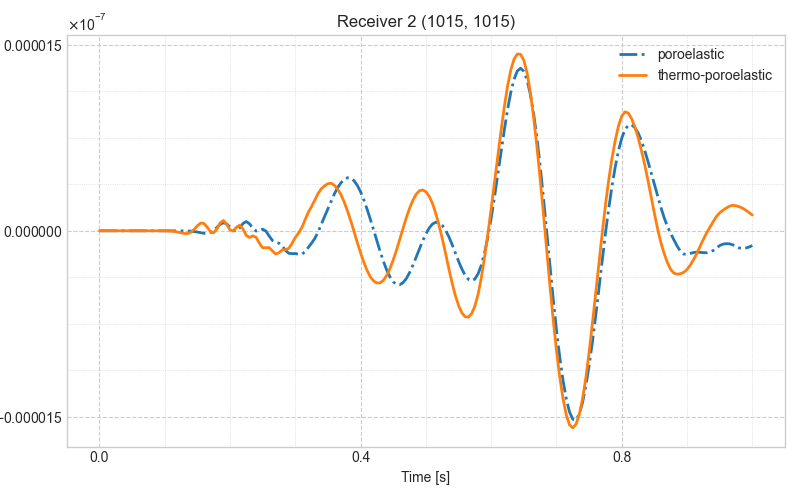}
\includegraphics[width=0.32\textwidth]{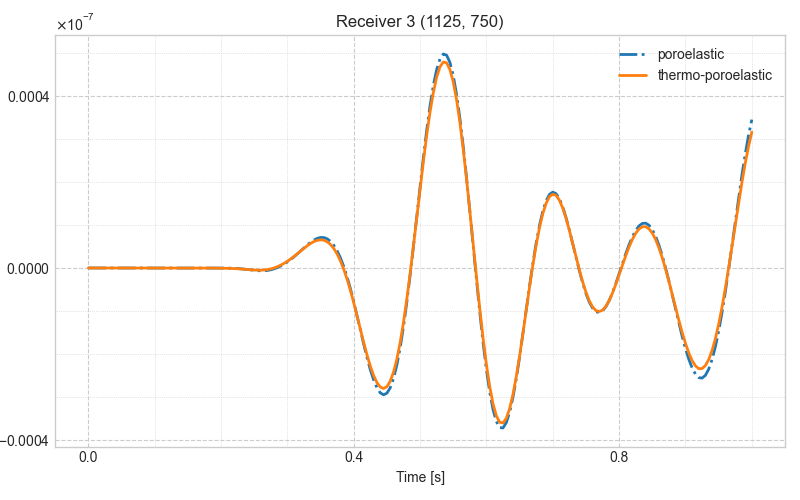}
\end{center}
\vspace{-4mm}
\caption{
The figure illustrates transient behaviour of the $y$-components of filtration velocity $\bq$ at the source receivers   $\boldsymbol{x}_{r1} = (750, 1125)$, $\boldsymbol{x}_{r2} = (1015, 1015)$ $\boldsymbol{x}_{r3} = (1125, 750)$ \unit{m} (left to right). We compare the thermo-poroelastic problem \eqref{eq:TP} shown with continuous lines to a purely poroelastic model problem represented by dashed lines.  We use the source term \eqref{Fshear}, and the parameter set \eqref{L3}, with $h=50$, $k=5$, and $\Delta t = 0.005$.}
\label{fig:heteroF}
\end{figure}

Figure~\ref{fig:heteroS} presents a comparative analysis of the temporal evolution of the $y$-components of solid velocity $u_2$ and filtration velocity $q_2$ at three receivers: $\mathbf{x}_{r1} = (750, 1125)$~m, $\mathbf{x}_{r2} = (1015, 1015)$~m, and $\mathbf{x}_{r3} = (1125, 750)$~m, corresponding to cardinal and diagonal directions from the source. The comparison contrasts the fully-coupled thermo-poroelastic model (continuous lines) with a purely poroelastic model neglecting thermal effects (dashed lines).

The solid velocity response reveals fundamentally different behavior depending on the receiver orientation relative to the source. At cardinal receiver locations ($\mathbf{x}_{r1}$ and $\mathbf{x}_{r3}$), positioned along the vertical and horizontal directions respectively, both models demonstrate reasonable agreement with only slight amplitude variations. The waveforms maintain similar overall characteristics, suggesting that thermal effects have limited influence on wave propagation in these preferential directions where compressional wave modes dominate.

In contrast, the diagonal receiver location ($\mathbf{x}_{r2}$) exhibits clear discrepancies between the two models. Compared to the purely poroelastic response, the thermo-poroelastic formulation yields different wave amplitudes and notable phase shifts. These enhanced coupling effects reflect the complex interaction between thermal and mechanical processes in regions where shear wave propagation dominates. 

The filtration velocity response displays in Figure~\ref{fig:heteroF} analogous features, though its amplitudes are several orders of magnitude lower than those observed in the solid velocity field.

The comparative analysis reveals that thermal effects exert significant influence on shear wave characterization. While purely poroelastic models provide adequate approximations for compressional wave behavior in cardinal directions, accurate prediction of complete wave field behavior in fluid-saturated thermoporoelastic media necessitates the inclusion of thermal coupling.

\subsection{Wave propagation in heterogeneous thermo-poroelastic media}

This subsection investigates wave propagation phenomena in a heterogeneous thermo-poroelastic medium to assess the method's capability in handling realistic geological configurations with spatially varying material properties. The computational domain $\Omega = (0, 1500) \times (0, 1500)$~m$^2$ is divided into two distinct regions along the vertical centerline: the left half ($x < 750$~m) retains the baseline material parameters from~\eqref{L3}, while the right half ($x \geq 750$~m) is characterized by modified thermo-poroelastic properties representing a stiffer, less compliant medium:
\begin{gather}\label{L4}
    \begin{split}
        \mu = 9 \times 10^9\,\unit{Pa}, \quad \lambda = 4 \times 10^9\,\unit{Pa}, \\
        c_0 = 4.1017\,\unit{Pa/K^2}, \quad b_0 = 1.3684 \times 10^{-5}\,\unit{K^{-1}}, \quad a_0 = 1.3684 \times 10^{-10}\,\unit{Pa^{-1}}, \\
        \alpha = 0.9514, \quad \beta = 2.4857 \times 10^{4}\,\unit{Pa/K}.
    \end{split}
\end{gather}

This heterogeneous configuration introduces several complex wave phenomena absent in homogeneous media. Figures~\ref{fig:veloSj} and~\ref{fig:veloTj} illustrate the temporal evolution of the mechanical and thermal fields at representative time instants $t = 0.1$, $0.3$, and $0.5$~s. The solid velocity magnitude $|\mathbf{u}|$ and its $y$-component $u_2$ (Figure~\ref{fig:veloSj}) demonstrate how the material interface disrupts the symmetric wave patterns observed in the homogeneous case. The wavefront interaction with the vertical interface generates complex reflection and transmission patterns, with clearly visible head waves propagating along the boundary between the two media.

The thermal field evolution (Figure~\ref{fig:veloTj}) also reveals pronounced sensitivity to material heterogeneity. The heat flux magnitude $|\mathbf{r}|$ and temperature $\theta$ distributions exhibit sharp discontinuities and enhanced boundary effects. The contrasting coupling coefficients $\beta$ between the two materials create preferential thermal pathways that channel thermal energy along the interface, resulting in pronounced boundary layer effects that persist throughout the simulation duration.

These results validate the robustness and accuracy of the proposed HDG method in handling complex heterogeneous configurations while simultaneously highlighting the  importance of thermal coupling in realistic geophysical applications. 

\begin{figure}[!ht]
	\begin{center}
	\includegraphics[scale=0.37]{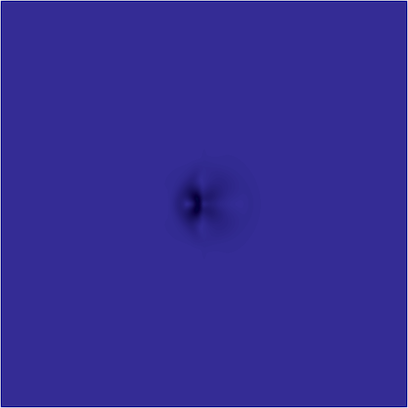}
	\includegraphics[scale=0.37]{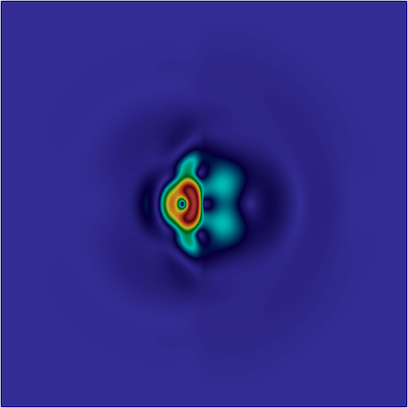}
	\includegraphics[scale=0.37]{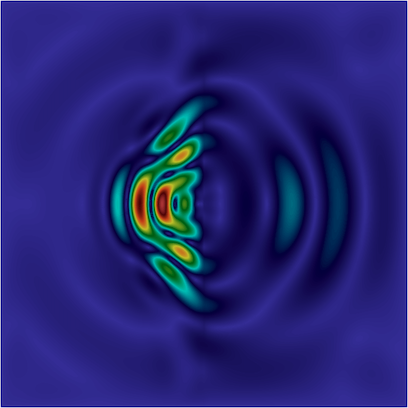}
	\\
	\includegraphics[width=0.5\textwidth]{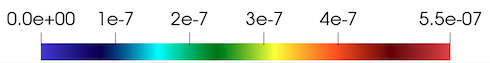}
	\\
	\includegraphics[scale=0.37]{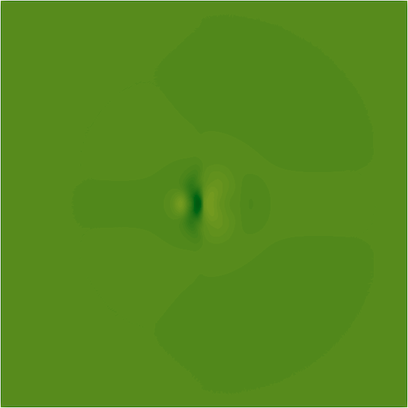}
	\includegraphics[scale=0.37]{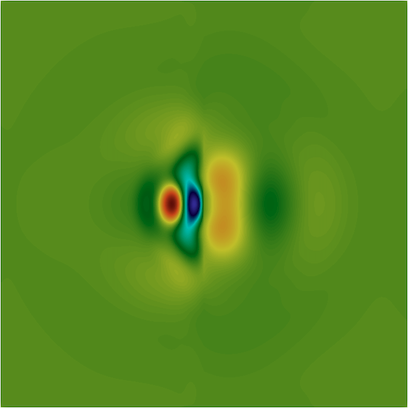}
	\includegraphics[scale=0.37]{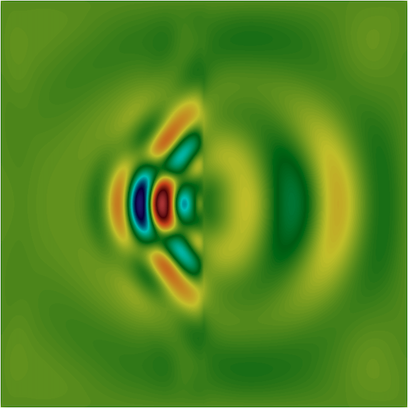}
	\\
	\includegraphics[width=0.5\textwidth]{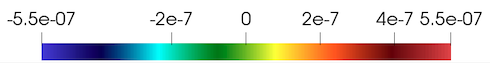}
	\end{center}
	\vspace{-4mm}
	\caption{Snapshots of the solid velocity magnitude $|\bu|$ (top row) and its $y$-component $u_2$ (bottom row) at times $t = 0.1$, $0.3$, and $0.5$~s (left to right panels) for the heterogeneous thermo-poroelastic problem. The computational domain features contrasting material properties with parameter set~\eqref{L3} in the left half ($x < 750$~m) and parameter set~\eqref{L4} in the right half ($x \geq 750$~m). The simulation employs the shear source term~\eqref{Fshear} with vanishing initial and homogeneous Dirichlet boundary conditions, discretized using $h = 50$~m, polynomial degree $k = 7$, and time step $\Delta t = 10^{-2}$~s.
}
	\label{fig:veloSj}
	\end{figure}

	\begin{figure}[!ht]
	\begin{center}
	\includegraphics[scale=0.35]{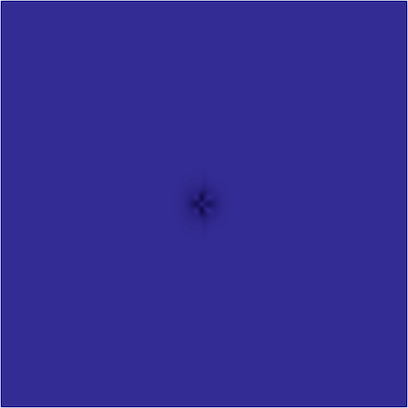}
	\includegraphics[scale=0.35]{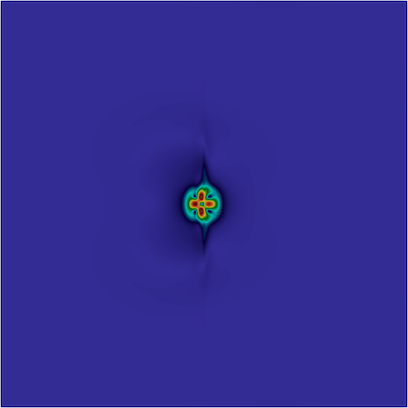}
	\includegraphics[scale=0.35]{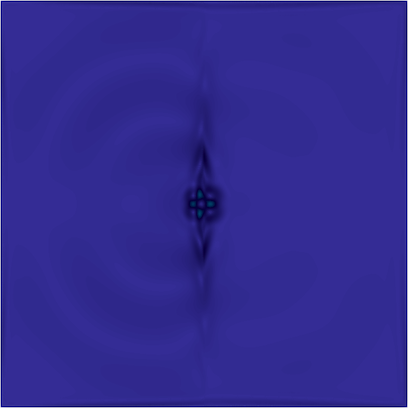}
	\\
	\includegraphics[width=0.5\textwidth]{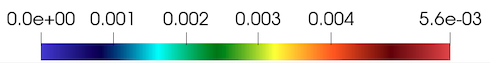}
	\\
	\includegraphics[scale=0.35]{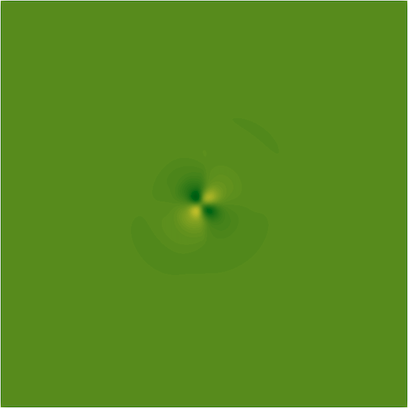}
	\includegraphics[scale=0.35]{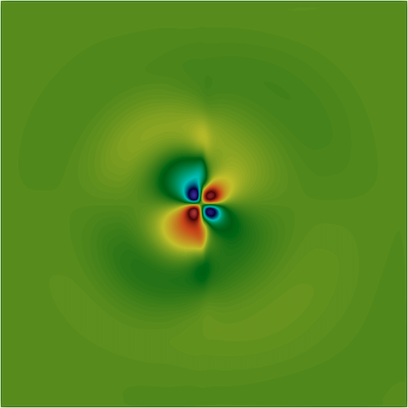}
	\includegraphics[scale=0.35]{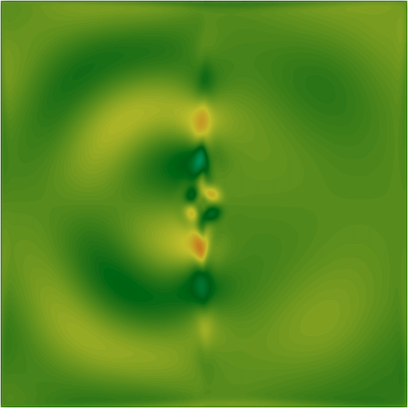}
	\\
	\includegraphics[width=0.5\textwidth]{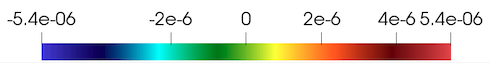}
	\end{center}
	\vspace{-4mm}
	\caption{
		Snapshots of the heat flux magnitude $|\br|$ (top row) and and the  temperature $\theta$ (bottom row) at times 0.1\,\unit{s}, 0.3\,\unit{s}, and 0.5\,\unit{s} (left to right panels) for the heterogeneous thermo-poroelastic problem. The computational domain features contrasting material properties with parameter set~\eqref{L3} in the left half ($x < 750$~m) and parameter set~\eqref{L4} in the right half ($x \geq 750$~m). The simulation employs the shear source term~\eqref{Fshear} with vanishing initial and homogeneous Dirichlet boundary conditions, discretized using $h = 50$~m, polynomial degree $k = 7$, and time step $\Delta t = 10^{-2}$~\unit{s}.
	}
	\label{fig:veloTj}
	\end{figure}

\section{Conclusions}\label{sec:conclusions}
This study presented a high-order hybridizable discontinuous Galerkin (HDG) formulation for fully dynamic thermo-poroelasticity. We established the well-posedness of the continuous problem using semigroup theory and demonstrated the method's optimal $hp$-error estimates. Numerical experiments, including manufactured solution tests and wave propagation benchmarks in heterogeneous media, confirmed the predicted convergence rates and validated the scheme's accuracy and efficiency in capturing coupled phenomena.

This HDG framework provides a robust and computationally effective tool for multiphysics simulations in geomechanics, biomechanics, and materials science. It accurately models wave propagation, including $P$-, $S$-, and thermal waves, making it valuable for predictive simulations.

Future work will explore extensions to non-linear fully-coupled thermo-poroelastic problems, specifically considering systems that include a non-linear convective transport term in the energy equation.

\appendix

\section{Finite Element Approximation Properties}\label{appendix}

This appendix recalls well-established $hp$-approximation properties, adapted to align with the specific notations and requirements of this work.

We begin with the following discrete trace inequality.
\begin{lemma}\label{TraceDG}
	There exists a constant $C>0$ independent of $h$ and $k$ such that
\begin{equation}\label{discTrace}
		 \norm{\tfrac{h^{\sfrac{1}{2}}_{\cF}}{k+1} q}_{0,\partial \cT_h} \leq C \norm*{ q}_{0, \cT_h}\quad \forall q \in \cP_{k}(\cT_h).
\end{equation}
\end{lemma}
\begin{proof}
	See \cite[Lemma 3.2]{meddahi2023hp}.
\end{proof}

For any integer $\ell\geq 0$ and element $K\in \cT_h$, $\Pi_K^\ell$ denotes the $L^2(K)$-orthogonal projection onto $\cP_{\ell}(K)$. The global projection $\Pi^\ell_\cT$ from $L^2(\Omega)$ onto $\cP_\ell(\cT_h)$ is defined element-wise by $(\Pi^\ell_\cT v)|_K = \Pi_K^\ell(v|_K)$ for all $K\in \cT_h$. Similarly, the global projection $\Pi_\cF^\ell$ from $L^2(\cF_h)$ onto $\cP_{\ell}(\cF_h)$ is defined face-wise by $(\Pi^\ell_\cF \hat v)|_{F} = \Pi_F^\ell(\hat v|_F)$ for all $F\in \cF_h$, where $\Pi^\ell_F$ is the $L^2(F)$-orthogonal projection onto $\cP_\ell(F)$. In the subsequent analysis, the notation $\Pi_\cT^\ell$ will also refer to the $L^2$-orthogonal projection onto $\mathcal{P}_\ell(\mathcal{T}_h, \mathbb{R}^{m\times n})$. It is important to note that the tensorial version of $\Pi_\cT^\ell$ inherently preserves matrix symmetry, as it is derived by applying the scalar operator component-wise. Likewise, $\Pi_\cF^\ell$ will also denote the $L^2$-orthogonal projection onto $\mathcal{P}_\ell(\mathcal{F}_h, \mathbb{R}^{m\times n})$.

We employ the following two approximation error estimates as key ingredients in deriving our asymptotic $hp$ error bounds. A detailed proof of these results can be found in \cite[Section 3]{meddahi2023hp} and the references therein.

\begin{lemma}[Approximation Property of $\Pi_\cT^k$]\label{lem:maintool}
	There exists a constant $C>0$ independent of $h$ and $k$ such that
	\begin{equation}\label{tool1}
		\norm{q - \Pi^k_\cT q}_{0,\cT_h} + \norm{ \tfrac{h^{\sfrac{1}{2}}_{\cF}}{k+1}(q - \Pi_\cT^k q)}_{0,\partial \cT_h}
		\leq
		C \tfrac{h_K^{\min\{ r, k \}+1}}{(k+1)^{r+1}} \norm{q}_{1+r,\Omega},
	\end{equation}
	for all $q \in H^{1+r}(\Omega)$, with $r\geq 0$.
\end{lemma}
\begin{proof}
	See \cite[Lemma 3.3]{meddahi2023hp}.
\end{proof}

Finally, we consider $\mathcal{U}^c$, a subspace of $\mathcal{U}$ comprising functions with continuous traces across the inter-element boundaries of $\cT_h$:
$$
\mathcal{U}^c \coloneq \set*{\ubv = [\bv \mid \bw \mid \bs] \mid \bv \in H_0^1(\Omega,\mathbb{R}^d),\ \bw, \bs \in H(\bdiv,\Omega)\cap H^r(\Omega,\mathbb{R}^d)}.
$$

\begin{lemma}[Approximation Property for Continuous Traces]\label{maintool2}
	There exists a constant $C>0$ independent of $h$ and $k$ such that
	\begin{equation}\label{tool2}
		\abs*{ \Big(\underline{\bu} - \Pi^{k+1}_{\cT} \underline{\bu}, \ \underline{\bu}|_{\partial \cF_h} - \Pi^{k+1}_{\cF} (\underline{\bu}|_{\partial \cF_h}) \Big)}_{\mathcal{U} \times \hat{\mathcal{U}}}
		\leq
		C \tfrac{h^{\min\{ r, k \}+1}}{(k+1)^{r+\sfrac{1}{2}}} \norm{\underline{\bu}}_{2+r,\Omega},
	\end{equation}
	for all $\underline{\bu} \in \mathcal{U}^c\cap H^{2+r}(\Omega,\bbR^{d\times 2})$, $r \geq 0$.
\end{lemma}
\begin{proof}
	The result is proved similarly to \cite[Lemma 3.4]{meddahi2023hp}.
\end{proof}

\bibliographystyle{plainnat}
\bibliography{cvBib.bib}

\end{document}